\documentclass[a4paper,fleqn]{cas-sc}

\usepackage[numbers,sort&compress]{natbib}

\usepackage{microtype}     
\usepackage{amsmath}
\usepackage{amssymb}
\usepackage{mathtools}
\usepackage{nicefrac}     
\usepackage{tabularx}
\usepackage{booktabs}       
\usepackage{multirow}
\usepackage{array}     
\usepackage{enumitem}
\usepackage{url}

\usepackage[caption=false,font={sf}]{subfig}
\DeclareRobustCommand{\mathbbm}[1]{\text{\usefont{U}{bbm}{m}{n}#1}}

\usepackage{tikz}
\usepackage{pgfplots}
\usepgfplotslibrary{groupplots}
\pgfplotsset{compat=1.17}  
\usetikzlibrary{
    arrows.meta,
    automata,
    decorations.pathreplacing,
    patterns,
    patterns.meta,
    positioning,
    shapes.geometric,
    calc,               
    matrix
}

\definecolor{RubineRed}{RGB}{209,0,86}
\definecolor{BlueGreen}{RGB}{0,158,169}
\definecolor{Brown}{RGB}{150,75,0}
\definecolor{ultramarineblue}{rgb}{0.05, 0.6, 0.96}
\colorlet{mycolor1}{teal!20!white}
\colorlet{mycolor2}{ultramarineblue!20!white}

\newcommand{\CP}{\mathrm{C_P}}  
\newcommand{\CS}{\mathrm{C_S}}  
\newcommand{\CA}{\mathrm{C_A}}

\newcommand{\E}{\mathbb{E}}                 
\newcommand{\var}{\mathbb{V}\mathrm{ar}}    
    
\newcommand{\Prob}{\mathbb{P}}              
\newcommand{\diff}{\,\mathrm{d}}

\newtheorem{theorem}{Theorem}
\newtheorem{lemma}[theorem]{Lemma}

\newtheorem{corollary}[theorem]{Corollary}
\newdefinition{definition}{Definition}
\newdefinition{assumption}{Assumption}
\newdefinition{remark}{Remark}
\newproof{pf}{Proof}

\ifcsmacro{proof}{}{\newenvironment{proof}{\begin{pf}}{\end{pf}}}

\begin{document}
\let\WriteBookmarks\relax
\def\floatpagepagefraction{1}
\def\textpagefraction{.001}

\shorttitle{Stochastic Bounds on System Time in $\mathrm{M^T/G/1}$ Queues}
\shortauthors{M. Rahnamania, M. Mandjes, F. Ashtiani}

\title[mode = title]{Analysis of Triggered Packet Streams: Stochastic Bounds on System Time under General Triggering Delays}

\author[1]{Mehran Rahnamania}[orcid=0009-0003-9098-5586]
\cormark[1]
\ead{mehran.r109@yahoo.com}

\author[2,3]{Michel Mandjes}[orcid=0000-0001-6783-4833]
\ead{m.r.h.mandjes@math.leidenuniv.nl}

\author[1]{Farid Ashtiani}[orcid=0000-0002-6955-1711]
\ead{ashtianimt@sharif.edu}

\affiliation[1]{organization={Department of Electrical Engineering, Sharif University of Technology},
            city={Tehran},
            country={Iran}}

\affiliation[2]{organization={Mathematical Institute, Leiden University},
            addressline={P.O. Box 9512}, 
            city={Leiden},
            postcode={2300 RA}, 
            country={The Netherlands}}

\affiliation[3]{organization={Korteweg-de Vries Institute for Mathematics, University of Amsterdam},
            city={Amsterdam},
            country={The Netherlands}}

\cortext[1]{Corresponding author}

\begin{abstract}
Many emerging communication systems increasingly involve devices that transmit packets in multiple causally dependent stages as part of an internal workflow, rather than waiting for receiver feedback. To analyze such systems, we recently introduced the $\mathrm{M^T/G/1}$ queue, in which primary customers arrive according to an exogenous Poisson process, and each triggers a secondary customer to join the queue after an independent random delay, called the \textit{triggering delay}. Our prior heuristic approximations assume identical exponential service times, and our matrix-analytic framework assumes exponential triggering delays. However, no exact analysis is available when both distributions are general.

In this paper, we derive stochastic bounds, in the usual stochastic order, on the system-time distribution of each customer class under general triggering-delay and service-time distributions. The service times within each primary--secondary pair may be generally dependent. Using sample-path coupling and level-crossing arguments, we show that the stationary workload is stochastically sandwiched between those of two classical queues: a two-class $\mathrm{M/G/1}$ queue with independent Poisson arrivals (lower bound), and a batch-arrival $\mathrm{M^{[2]}/G/1}$ queue where each primary--secondary pair arrives simultaneously (upper bound). Notably, these bounds do not depend on the triggering-delay distribution, and they are asymptotically tight in the sense that the stationary workload converges in distribution to that of the lower- and upper-bound queues as the triggering delays diverge to infinity and converge to zero in probability, respectively. We translate these workload bounds into system-time bounds for each customer class. Simulation experiments support these bounds.
\end{abstract}


\begin{keywords}
Queueing theory \sep Triggered-arrival queues \sep Sample-path coupling \sep Stochastic ordering \sep Level-crossing
\end{keywords}

\maketitle

\section{Introduction}
\label{sec:intro}

Many emerging communication systems increasingly involve devices that transmit packets in multiple causally dependent stages as part of an internal workflow, rather than waiting for receiver feedback. Specifically, in the simplest setting, after transmitting a primary packet, a user might transmit a secondary packet after a potentially random delay determined by its own workflow, without waiting for feedback or acknowledgment~\cite{Rahnamania2024CorrelatedArrivals}. Such traffic patterns arise naturally across a wide range of applications. In Industrial Internet of Things (IIoT) systems, for example, a lightweight alarm may be transmitted immediately after an anomaly is detected, followed by a more detailed report after a local sensing or processing delay~\cite{s23073544}. In ultra-reliable low-latency communication (URLLC), blind packet repetitions may likewise be transmitted at predetermined offsets to enhance reliability without incurring feedback delays~\cite{9174916}. The same multi-stage feedback-free transmission pattern can also arise in applications involving AI agents. For instance, they may send responses in several parts, such as the text first and the corresponding audio thereafter~\cite{liu-moq-live-agent-interaction-02}.

These examples illustrate a common communication scenario in which numerous independent users operate concurrently over a shared communication link, and the transmission epochs of packets generated by each user are causally dependent. When transmissions from all sources are aggregated at a gateway, the resulting aggregate arrival process exhibits complex interarrival dependence. To analyze such systems, we recently introduced the $\mathrm{M^T/G/1}$ queue, in which primary customers arrive according to an exogenous Poisson process, and each primary customer triggers exactly one secondary customer to join the queue after a random delay, referred to as the \textit{triggering delay}~\cite{Rahnamania2024CorrelatedArrivals,Rahnamania2025MatrixAnalytic} (see Fig.~\ref{fig:model:system}). It is important to note that the Poisson assumption for primary arrivals is motivated by the fact that the superposition of many independent point processes (e.g., those generated by a large number of sensors or users) is well approximated by a Poisson process.

Standard queueing frameworks and classical methodologies are not directly applicable to this queue. The causal dependence between primary and secondary arrivals in the aggregate arrival stream precludes the use of standard analytical methods used for queueing systems with renewal arrival processes, including embedded Markov chains and Lindley's integral equation for $\mathrm{M/G/1}$ and $\mathrm{G/M/1}$ queues~\cite{Kleinrock1975Volume1,asmussen2003applied}.

Frameworks such as the Batch Markovian Arrival Process (BMAP) and the Markov-Modulated Poisson Process (MMPP) can model dependent arrivals through finite-state Markovian modulation~\cite{FISCHER1993149,chakravarthy2001batch,Latouche1999MatrixAnalytic,neuts1981matrix}. In these models, dependence is captured by a finite-dimensional background Markov chain, which is precisely what gives these types of queues much of their tractability. By contrast, in the $\mathrm{M^T/G/1}$ queue, the arrival epoch of each secondary customer depends on that of a primary customer through a random generally distributed triggering delay. A secondary customer may arrive long after its corresponding primary customer and may be interleaved with primary--secondary arrivals from other sources. Consequently, when triggering delays are generally distributed, an exact Markovian representation would typically require tracking information about all currently active triggering delays; thus, the effective state space is not \textit{finite-dimensional} in any straightforward sense.

Self-exciting point processes, particularly Hawkes processes, provide a widely studied framework for modeling dependent and clustered arrivals~\cite{Hawkes1971,Daw2018}. Nevertheless, they do not capture the arrival structure considered here. In a Hawkes process, an arrival alters the future conditional intensity through a self- or mutually exciting mechanism. Under its branching interpretation, an arrival may generate multiple offspring, which may themselves generate further arrivals, resulting in multigenerational clusters. In our model, by contrast, each primary arrival triggers exactly one secondary arrival after a random triggering delay, and secondary arrivals generate no further arrivals. Accordingly, the dependence arises from an explicit one-to-one connection between primary and secondary arrival epochs, rather than from a general self-excitation mechanism, making the arrival structure different from Hawkes-type clustering.

Finally, G-networks provide a framework for modeling concurrent movement in a queueing network through positive and negative signals and admit product-form solutions under quasi-reversibility conditions~\cite{6796323,ec60910f-c7c5-39b2-8655-659b6b05ba84,1130000795627348864}. However, the product-form results based on quasi-reversibility~\cite[Chs.~3--4]{1130000795627348864} do not directly apply to the one-to-one primary--secondary dependence considered here without introducing a substantial approximation via the ``additional departure'' mechanism required to represent this triggered arrival using positive signals.

To address these analytical challenges, our initial investigations of the $\mathrm{M^T/G/1}$ queue proceeded in three stages. First, we developed a heuristic approximation for the mean system times under general triggering-delay distributions, assuming that the service-time distributions of primary and secondary customers are independent and follow the same exponential distribution~\cite{Rahnamania2024CorrelatedArrivals}. Subsequently, we refined our analysis by leveraging a structural property of secondary arrivals to develop a more accurate heuristic approximation of the system-time distributions for both customer classes under phase-type triggering delays and independent, identically distributed exponential service times~\cite{11571174}. More recently, we derived a matrix-analytic characterization with controllable truncation error for the stationary workload and system-time distributions when triggering delays follow an exponential distribution and service times are independent, with class-specific general distributions~\cite{Rahnamania2025MatrixAnalytic}.

Extending these analyses to triggering delay and service time with general distributions remains an open problem and is not readily tractable within the same frameworks. Both heuristic approximations rely specifically on \textit{identical exponential} service times for the two customer classes~\cite{Rahnamania2024CorrelatedArrivals,11571174}. The matrix-analytic solution, in turn, relies on exponential triggering delays. The obstacle is structural. Under general triggering delays, an exact Markovian characterization would generally need to track the residual triggering delay of every pending secondary customer, i.e., every pair whose primary customer has already arrived but whose secondary customer has not yet arrived. Because the number of such pending secondary customers is unbounded, this description involves a vector of continuous residual triggering delays of random, unbounded length, and the resulting state space can become infinite-dimensional rather than the finite- or level-dependent structures usually exploited in standard matrix-analytic methods.

One approach is to approximate general triggering-delay distributions using phase-type (PH) distributions within a matrix-analytic framework. However, this approach can substantially increase the complexity of the state space. In particular, with an $m$-phase PH representation and independent service times, one must track the number of active triggering delays in each phase, resulting in an $m$-dimensional count vector with unbounded coordinates, in addition to the workload. When service times within a primary--secondary pair are dependent, the pair-specific dependence must also be tracked. Thus, the analytical obstacle is an unbounded, multidimensional level structure rather than a simple one-dimensional queue. This makes finite-state truncation and rigorous error control challenging~\cite{ABATE2021102207,SOMASHEKAR2023102356}. Moreover, constructing the corresponding generator matrix is highly complex, as it must simultaneously account for: (i) Poisson arrivals of primary customers; (ii) the initiation of new triggering delays; (iii) phase transitions of all ongoing triggering delays; (iv) the completion of triggering delays that lead to the arrival of secondary customers; and (v) the residual service time of the customer currently in the server.

This paper adopts a distinct approach. Rather than seeking a closed-form exact characterization of the system-time distributions or employing high-dimensional state-space truncations for generally distributed triggering delays, we establish rigorous stochastic lower and upper bounds, in the usual stochastic order ($\le_{\mathrm{st}}$), for the system-time distributions of the $\mathrm{M^T/G/1}$ queue. Utilizing the auxiliary customer class introduced in~\cite{11571174}, we define two embedded workload processes: a reference process initiated by an auxiliary arrival and a target process initiated by a secondary arrival. Coupling these processes in a common probability space enables us to track workload evolution over a shared triggering-delay interval and facilitates a sample-path comparison. Consequently, we show that an arriving secondary customer observes a workload that is stochastically larger than the steady-state workload but bounded above by that workload plus the service time of its corresponding primary customer. Combining this result with a level-crossing argument, we derive rigorous bounds in the usual stochastic order for the stationary workload and system-time distributions. These bounds are independent of the triggering-delay distribution. In particular, we show that the stationary workload $V_{\mathrm{M}}$ of the $\mathrm{M^T/G/1}$ queue is stochastically \textit{sandwiched} between the stationary workloads of two tractable queues:
\begin{equation}
    \label{eq1}
    V_{\mathrm{L}} \le_{\mathrm{st}} V_{\mathrm{M}} \le_{\mathrm{st}} V_{\mathrm{U}},
\end{equation}
where $V_{\mathrm{L}}$ denotes the stationary workload of a standard two-class $\mathrm{M/G/1}$ queue in which primary and secondary customers arrive as independent Poisson processes (the lower bound), while $V_{\mathrm{U}}$ corresponds to an $\mathrm{M^{[2]}/G/1}$ queue with batch arrivals of size two, where each primary--secondary pair arrives simultaneously in a batch (the upper bound). The service times within each batch maintain the dependence structure of the original system. This stochastic sandwich, which holds for general triggering-delay distributions and any dependence between the service requirements within each pair, constitutes a key theoretical contribution of the paper. We then use the PASTA property to translate \eqref{eq1} into stochastic bounds on the system time of primary customers, while the corresponding bound for secondary customers follows from our sample-path analysis. Moreover, properly using our sample-path coupling, we show that the bounds become asymptotically tight in the sense that the stationary workload converges in distribution to the lower-bound workload as the triggering delays diverge in probability to infinity and to the upper-bound workload as the triggering delays converge in probability to zero. Finally, using the intuition derived from our analysis, we provide a heuristic engineering guideline that offers an intuitive threshold for determining when either the lower or the upper bound is expected to provide a practical approximation.
Table~\ref{tab:comparison_prior_works} summarizes the key distinctions between this work and previous studies.

\begin{table*}[t]
\centering
\caption{Comparison of prior works with the proposed $\mathrm{M^T/G/1}$ stochastic bounding framework.}
\label{tab:comparison_prior_works}
\scriptsize
\renewcommand{\arraystretch}{1.0}
\setlength{\tabcolsep}{6pt}

\begin{tabular}{@{}lcccc@{}}
\toprule
\textbf{Reference} 
& \textbf{Triggering delay} 
& \multicolumn{2}{c}{\textbf{Service times}} 
& \textbf{Main result} \\
\cmidrule(lr){3-4}
& 
& \textbf{Distribution} 
& \textbf{Dependence} 
& \\
\midrule
\cite{Rahnamania2024CorrelatedArrivals} 
& General 
& Identical exponential 
& Independent 
& Heuristic mean system-time approximation \\

\cite{11571174} 
& Phase-type 
& Identical exponential 
& Independent 
& Heuristic system-time distribution approximation \\

\cite{Rahnamania2025MatrixAnalytic} 
& Exponential 
& General, class-dependent 
& Independent 
& Matrix-analytic method with controllable truncation error \\

\textbf{This work} 
& General 
& General, class-dependent 
& General dependence
& Stochastic lower and upper bounds on system times \\
\bottomrule
\end{tabular}
\end{table*}

The main contributions of this paper are summarized as follows:
\begin{itemize}
\item \textit{Sample-Path Coupling Framework}: We develop a sample-path coupling construction involving two auxiliary components to relate the workload observed by secondary customers in the $\mathrm{M^T/G/1}$ queue to its stationary workload distribution. This construction establishes the key stochastic ordering in Corollary~\ref{cor:ordering}, which underpins our subsequent analysis.

\item \textit{Workload Bounds}: We derive rigorous lower and upper bounds in the usual stochastic order for the stationary workload distribution using a level-crossing argument. These bounds hold for general triggering-delay distributions and general dependence between service times within each pair, provided that the service-time second moments are finite. The lower bound is given by the stationary workload of a two-class $\mathrm{M/G/1}$ queue, in which primary and secondary customers arrive according to independent Poisson processes. The upper bound is given by the stationary workload of an $\mathrm{M^{[2]}/G/1}$ queue, in which each primary customer and its corresponding secondary customer arrive simultaneously as a batch. These bounds are independent of the triggering-delay distribution. We further demonstrate that the bounds are asymptotically tight in the sense that the stationary workload converges in distribution to that of the lower-bound queue as triggering delays diverge in probability to infinity, and to that of the upper-bound queue as triggering delays converge in probability to zero (see Theorems~\ref{thm:asymp_large_delay} and~\ref{thm:asymp_zero_delay}).

\item \textit{System-Time Distribution Bounds}: We translate the workload bounds into corresponding stochastic bounds on the system-time distributions of both customer classes, while accommodating general service-time distributions and dependence between the service requirements of each primary--secondary pair (see Theorem~\ref{thm:system_time_bounds}).

\item \textit{Heuristic Engineering Approximation}: We provide a heuristic rule of thumb that compares the mean triggering delay with the expected times for the lower-bound queue to become empty and the upper-bound queue to reach a specified workload level to determine when either bound may serve as a practical approximation, and we validate this rule numerically.

\end{itemize}

The remainder of this paper is organized as follows. Section~\ref{sec:model} introduces the system model and formalizes the queueing dynamics. Section~\ref{sec:ph} presents the core sample-path coupling framework. Section~\ref{sec:bounds} derives the stochastic bounds, establishes their asymptotic tightness, translates them into bounds on system time, and provides heuristic engineering guidelines for selecting an appropriate bound for practical approximation. Section~\ref{sec:numerical} presents the numerical results, and Section~\ref{sec:conclusion} concludes the paper.

\section{System Model}
\label{sec:model}

In this section, we formally define the $\mathrm{M^T/G/1}$ queue and distinguish between the arrival processes of the two customer classes in terms of the applicability of the PASTA property.

\subsection{Arrival Process and Queue Dynamics}
\label{subsec:model:description}

\begin{figure}
    \centering
    \begin{tikzpicture}[scale=1]
    
    \foreach \x in { 0.65, 0.925, 1.2,1.475}
    \draw[thick,black,fill=mycolor1] (\x+0.5, 0.075) rectangle (\x+0.7, 0.925);

    \draw[->, line width=0.55mm, RubineRed] (2.7, 0.7) -- (3.7, 0.7);
    \draw[->, line width=0.55mm,BlueGreen, densely dashed] (2.7, 0.4) -- (3.7, 0.4);
    \node[BlueGreen] at (-4.15,0.5-0.15) {$\CS$};
    \node[RubineRed] at (-4.15,1.2) {$\CP$};
    \node[black] at (-1.25,2.15-0.1) {\scriptsize triggering delay};
    \node[black] at (-1.25,-0.9) {\scriptsize triggering delay};

    \draw[thick,black,fill=RubineRed]  (-0.2,0.8-0.2) rectangle (0,1.6-0.2);
    \draw[thick,black,fill=BlueGreen]  (-0.2-2.2,0.1-0.4) rectangle (0-2.2,0.8-0.3);

    \draw[thick,black,pattern=crosshatch, pattern color=RubineRed]  (-0.2-0.9,0.8-0.2) rectangle (0-0.9,1.6-0.2);
    \draw[thick,black,pattern=crosshatch, pattern color=BlueGreen]  (-0.2-1.5,0.1-0.4) rectangle (0-1.5,0.8-0.3);

    \draw[thick,black, pattern=checkerboard, pattern color=RubineRed]  (-2.6,0.6) rectangle (-2.4,1.4);
    \draw[thick,black, pattern=checkerboard, pattern color=BlueGreen]  (-3.7,0.1-0.4) rectangle (-3.5,0.8-0.3);
    
    \draw[->,  line width=0.55mm, RubineRed] (0.05, 1.1-0.2)  to[in=180, out=0] (0.8, 0.7);
    \draw[-,  line width=0.55mm, RubineRed] (-0.85, 1.1-0.2) -- (-0.27, 1.1-0.2);
    \draw[-,  line width=0.55mm, RubineRed] (-2.35, 1.1-0.2) -- (-1.175, 1.1-0.2);
    \draw[-,  line width=0.55mm, RubineRed] (-4, 1.1-0.2) -- (-2.65, 1.1-0.2);

    \draw[densely dashed,->,  line width=0.55mm, BlueGreen] (-1.4, 0.4-0.3) -- (0, 0.1)--(0, 0.4-0.3) to[in=180, out=0] (0.8, 0.3);
    \draw[densely dashed,-,  line width=0.55mm, BlueGreen] (-2.15, 0.4-0.3) -- (-1.75, 0.4-0.3);
    \draw[densely dashed,-,  line width=0.55mm, BlueGreen] (-3.45, 0.4-0.3) -- (-2.45, 0.4-0.3);
    \draw[densely dashed,-,  line width=0.55mm, BlueGreen] (-4, 0.4-0.3) -- (-3.75, 0.4-0.3);

      \draw[-, line width=0.2mm, densely dotted, black] (-0.1, 0.6)  -- (-0.1, -0.6);
    \draw[<->, line width=0.5mm, densely dotted, black] (-2.3, -0.6)  -- (-0.1, -0.6);
    \draw[-, line width=0.2mm, densely dotted, black] (-2.3, -0.3)  -- (-2.3, -0.6);

    \draw[<->, line width=0.5mm, densely dotted, black] (-1.6, 1.9-0.2)  -- (-1, 1.9-0.2);
    \draw[<->, line width=0.5mm, densely dotted, black] (-3.6, 1.9-0.2)  -- (-2.5, 1.9-0.2);
    \draw[-, line width=0.2mm, densely dotted, black] (-3.6, 1.7)  -- (-3.6, 0.5);
    \draw[-, line width=0.2mm, densely dotted, black] (-2.5, 1.7)  -- (-2.5, 1.4);
    
    \draw[-, line width=0.2mm, densely dotted, black] (-1.6, 1.9-0.2)  -- (-1.6, 0.5);
    \draw[-, line width=0.2mm, densely dotted, black] (-1, 1.9-0.2)  -- (-1, 1.5-0.1);

    \draw[line width=0.55mm,fill=mycolor2]  (2.75,0.5) circle [radius=0.5];
    \draw[line width=0.55mm] (0.5,0) -- (2.25,0) -- (2.25,1) -- (0.5, 1);
    
\end{tikzpicture}
    \caption{Illustration of the $\mathrm{M^T/G/1}$ queue. Primary customers ($\CP$) arrive according to a Poisson process, and each triggers a corresponding secondary customer ($\CS$) to join the queue after a random triggering delay.}
    \label{fig:model:system}
\end{figure}
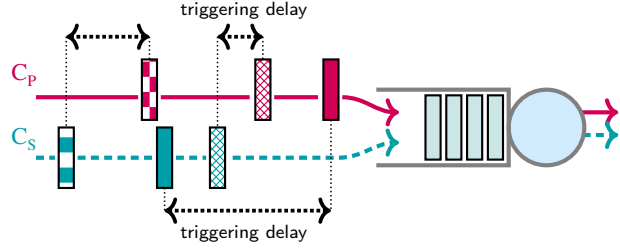

We consider a single-server queueing system with infinite buffer capacity and first-come-first-served (FCFS) service discipline. We denote this queueing system by $\mathrm{M^T/G/1}$, where ``T'' refers to the triggering mechanism described below. As illustrated in Fig.~\ref{fig:model:system}, the system consists of two customer classes:
\begin{itemize}
    \item \emph{Primary customers} ($\CP$): These customers arrive according to an exogenous Poisson process with rate $\lambda$.

    \item \emph{Secondary customers} ($\CS$): Each primary customer, upon arriving at the queue, triggers a corresponding secondary customer to join the queue after a random delay, referred to as the \emph{triggering delay}. The triggering delays $\{Y_i\}_{i\ge 1}$ are i.i.d. random variables distributed identically to a generic random variable $Y$, whose distribution function is $F_Y(y)\coloneqq \Prob(Y\le y)$ and whose mean is finite, i.e., $\E[Y] < \infty$.
\end{itemize}

Equivalently, let $\{T_i\}_{i\ge 1}$ denote the sequence of primary arrival epochs. The corresponding secondary arrival epochs are then given by $\{T_i+Y_i\}_{i\ge 1}$. The service requirements of the $i$th primary customer and its corresponding secondary customer are denoted by the pair $(X_{\CP,i}, X_{\CS,i})$. We assume that the sequence of service-requirement pairs $\{(X_{\CP,i},X_{\CS,i})\}_{i\ge 1}$ is i.i.d. and follows the distribution of a generic random pair $(X_{\CP},X_{\CS})$ with joint distribution function
\[
    B_{\CP,\CS}(x_1,x_2)\coloneqq \Prob(X_{\CP}\le x_1, X_{\CS}\le x_2).
\]
The corresponding marginal distribution functions are
\[
    B_{\CP}(x)\coloneqq \Prob(X_{\CP}\le x),
    \qquad
    B_{\CS}(x)\coloneqq \Prob(X_{\CS}\le x),
\]
with tail distribution functions
\[
    \overline B_{\CP}(x)\coloneqq 1-B_{\CP}(x),
    \qquad
    \overline B_{\CS}(x)\coloneqq 1-B_{\CS}(x).
\]
We assume that $X_{\CP}$ and $X_{\CS}$ have finite first and second moments. Furthermore, the primary arrival epochs $\{T_i\}_{i\ge 1}$, the triggering delays $\{Y_i\}_{i\ge 1}$, and the service-requirement pairs $\{(X_{\CP,i},X_{\CS,i})\}_{i\ge 1}$ are mutually independent.

Because each $\CP$ generates exactly one $\CS$, the total long-run average arrival rate into the system is $2\lambda$. However, the aggregate arrival process is generally neither a Poisson process nor a renewal process~\cite{Rahnamania2025MatrixAnalytic}. For this reason, throughout this paper, the term ``rate'' refers to the long-run average rate.

For stability, we require
\begin{equation}\label{eq:model:stability}
    \rho=\lambda\big(\E[X_{\CP}]+\E[X_{\CS}]\big)<1,
\end{equation}
where $\rho$ denotes the traffic intensity~\cite{asmussen2003applied}. Under condition \eqref{eq:model:stability}, Loynes' stability theorem for stationary and ergodic queues ensures that the system admits a unique stationary workload distribution. The analysis in this paper is conducted in this stationary regime, and the terms \emph{stationary} and \emph{steady-state} are used interchangeably. 

The two arrival streams differ in terms of the applicability of the PASTA property. For ease of reference, we recall the relevant definitions from \cite{Rahnamania2024CorrelatedArrivals} below.

\subsection{Applicability of PASTA}
\label{subsec:model:pasta}

\begin{definition}[\emph{Causally independent} Poisson arrival process]
\label{def:exp:causal_indep}
An arrival process is a \emph{causally independent Poisson process} if it is Poisson and its arrival epochs are independent of the history and current workload of the queue, including the arrival epochs of other customer classes that affect the queue workload.
\end{definition}

Arrivals of primary customers satisfy the lack-of-anticipation assumption. Accordingly, the PASTA property applies, and arriving $\CP$ customers observe the stationary workload distribution.

Although the marginal arrival process for $\CS$ customers is a Poisson process, each secondary arrival epoch equals its corresponding primary arrival epoch plus a triggering delay. Consequently, an arrival of a $\CS$ customer conveys information about a prior arrival of a $\CP$ customer that has already influenced the workload. Therefore, the secondary arrival process is not causally independent, and the PASTA property does not directly apply to $\CS$ customers.

\section{Stochastic Ordering Principles}
\label{sec:ph}

\begin{figure}
    \centering
        \begin{tikzpicture}[scale=0.9]
        \draw[->, line width=0.55mm, RubineRed] (-0.725, 3.5) -- (-0.725, 1.25) to[in=180, out=-90] (0.25, 0.75);
        \draw[->, line width=0.55mm, RubineRed] (2.6, 0.75) -- (3.5, 0.75);

        \draw[-, very thick, black] (-2.5, 0.5) -- (-2.25, 0.5);
        \draw[-, very thick, black] (-2.5, 0.525) -- (-2.25, 0.525);
        \draw[-, very thick, black] (-2.5, 0.55) -- (-2.25, 0.55);
        \draw[-, very thick, black] (-2.5, 0.475) -- (-2.25, 0.475);
        \draw[-, very thick, black] (-2.5, 0.45) -- (-2.25, 0.45);

        \draw[densely dashed,->, line width=0.55mm, BlueGreen] (-0.8, 3.5) -- (-0.8, 3.4)  to[in=90, out=-90] (-3.5,2.3) -- (-3.5, 1)to[in=180, out=-90] (-2.75, 0.5);
        \draw[densely dotted,->, line width=0.55mm, black]  (-3.75, 0.25) -- (-2.75, 0.25);
        \draw[densely dashed,->, line width=0.55mm, BlueGreen] (-1.25, 0.5) -- (0.25, 0.5);
        \draw[densely dotted,->, line width=0.55mm, black] (-1.25, 0.25) -- (-0.75, 0.25) to[in=90, out=0] (-0.25, -0.5);
        \draw[densely dashed,->, line width=0.55mm, BlueGreen] (2.75, 0.5) -- (3.5, 0.5);

        \foreach \x in {0.65, 0.925, 1.2,1.475}
        \draw[thick,black,fill=mycolor1] (\x, 0.075) rectangle (\x+0.2, 0.925);

        \draw[ thick,pattern=checkerboard, pattern color=RubineRed] (-0.7,3.6) rectangle (-0.5,4.4) ;
        \draw[ thick, black,fill=BlueGreen] (-0.8,3.6) rectangle (-1,4.4) ;
        \draw[ thick, pattern=crosshatch, pattern color=black] (-3.9,0.7) rectangle (-4.1,-0.1) ;

        \node[RubineRed] at (-0.2,3.75){$\CP$};
        \node[RubineRed] at (-0.5,2.75){$\lambda$};
        \node[BlueGreen] at (-1.35,3.75){$\CS$};
        \node[black] at (-4.4,0.25){$\CA$};
        \node[BlueGreen] at (-1.3,3.1){$\lambda$};

        \draw[line width=0.55mm,fill=mycolor2]  (2.25,0.5) circle [radius=0.5];

        \draw[line width=0.55mm,fill=mycolor1] (-2.25,-0.35) arc[start angle=180, end angle=360,radius=0.5cm] --(-1.25,-0.35) -- (-1.25,2.2) -- (-2.25, 2.2) -- (-2.25,-0.35);

        \draw[line width=0.55mm,fill=mycolor2]  (-1.75,2.2) circle [radius=0.5];

        \draw[line width=0.55mm,fill=mycolor2]  (-1.75,0.5) circle [radius=0.5];

        \node at (-1.75,-0.25) {\huge\vdots};
        \node at (-1.75, 1.45) {\huge\vdots};

        \node at (1.5, -1.25) {$\mathrm{Q_M}$: $\mathrm{M^T/G/1}$};
        \node at (-2.5, -1.25) {$\mathrm{Q_A}$: $\mathrm{M/G/\infty}$};
        \draw[line width=0.55mm] (-3,0) -- (-2.5,0) -- (-2.5,1) -- (-3, 1);
        \draw[line width=0.55mm] (0,0) -- (1.75,0) -- (1.75,1) -- (0, 1);

    \end{tikzpicture}
    \caption{The auxiliary $\mathrm{M/G/\infty}$ queue ($\mathrm{Q_A}$) models the triggering delays for $\CS$ customers before arriving at the main queue ($\mathrm{Q_M}$). To facilitate the analysis, we introduce a class of auxiliary customers ($\CA$) that experience the same triggering delays but exit the system entirely.}
    \label{fig:ph:aux_components}
\end{figure}

As argued in Section~\ref{sec:intro}, no general exact solution is currently available for the $\mathrm{M^T/G/1}$ queue. We therefore pursue a stochastic bounding approach. Accordingly, this section derives stochastic bounds on the workload distribution observed by arriving $\CS$ customers, whose arrival process does not satisfy the PASTA property. To this end, we first introduce foundational auxiliary constructions and two embedded workload processes. We place these processes on the same probability space and construct a sample-path coupling that yields the key stochastic ordering underlying the analysis. The auxiliary construction preserves the original queueing dynamics while enabling a direct comparison between the workload distribution observed by arriving $\CS$ customers and the stationary workload distribution of the queue. This comparison provides the basis for deriving stochastic upper and lower bounds on the system-time distributions of both customer classes in Section~\ref{sec:bounds}.

The construction relies on two auxiliary components, as illustrated in Fig.~\ref{fig:ph:aux_components}.
\begin{itemize}
    \item A conceptual $\mathrm{M/G/\infty}$ queue, referred to as the \textit{auxiliary queue}, $\mathrm{Q_A}$. This queue does not alter the system dynamics but provides a clear framework for tracking triggering delays. As shown in Fig.~\ref{fig:ph:aux_components}, when a $\CP$ arrives at the \textit{main queue}, $\mathrm{Q_M}$, its corresponding $\CS$ simultaneously arrives at $\mathrm{Q_A}$. The ``service time'' in $\mathrm{Q_A}$ is exactly the random triggering delay $Y$. Upon service completion at $\mathrm{Q_A}$, the $\CS$ arrives at $\mathrm{Q_M}$.

    \item \textit{Auxiliary customers}, denoted by $\CA$. As illustrated in Fig.~\ref{fig:ph:aux_components}, these $\CA$ customers arrive at $\mathrm{Q_A}$ according to a causally independent Poisson process with an arbitrary positive rate $\lambda_{\CA} > 0$. Upon completion of their triggering delay, each $\CA$ customer leaves the system entirely rather than joining $\mathrm{Q_M}$. Since $\mathrm{Q_A}$ is an infinite-server queue, customers neither wait nor interact within it. Consequently, introducing $\CA$ customers does not affect the evolution of $\mathrm{Q_M}$; it serves solely as an analytical device.
\end{itemize}

Our analytical strategy proceeds as follows. Using the auxiliary customer class, we construct two embedded workload processes: a \emph{reference process} (RP) initiated by the arrival of a tagged $\CA$ at $\mathrm{Q_A}$, and a \emph{target process} (TP) initiated by the arrival of a tagged $\CS$ at $\mathrm{Q_A}$. We first identify the exact initial states of both processes at $t=0^+$ (that is, the moment immediately after the respective customer arrives at $\mathrm{Q_A}$). We then couple the RP and TP on a common probability space so that they evolve under identical realizations of all underlying random primitives, allowing us to track their workload sample paths over the shared triggering-delay interval $(0, Y)$. Finally, we compare the workload of $\mathrm{Q_M}$ in each process at time $Y^-$ (that is, the moment immediately before the respective customer departs from $\mathrm{Q_A}$). This comparison reveals the precise sample-path relationship between the workload of $\mathrm{Q_M}$ seen by the tagged $\CS$ and that observed by the tagged $\CA$ upon the completion of their respective triggering delays.

By formalizing this sample-path comparison, we derive Theorem~\ref{thm:ph:conditional_dist}, which establishes an exact distributional relationship between the workload seen by a $\CS$ upon arriving at $\mathrm{Q_M}$ and the true steady-state workload of $\mathrm{Q_M}$ (which corresponds to the workload seen by a $\CA$ in our analysis). As a consequence, Corollary~\ref{cor:ordering} shows that in the $\mathrm{M^T/G/1}$ queue, the workload observed by an arriving $\CS$ is stochastically sandwiched between the steady-state workload and the steady-state workload plus its corresponding $\CP$ service requirement. These results provide the foundation for the upper and lower bounds on the workload and system-time distributions of each customer class developed in the subsequent section.

\subsection{Embedded Processes and Coupling Construction}
\label{subsec:ph:init_state}

Let $V_{\mathrm{M}}(t)$ denote the workload process of $\mathrm{Q_M}$ at time $t$, whose sample paths are right-continuous with left limits (c\`adl\`ag). The state of $\mathrm{Q_A}$ is represented by the vector of residual lifetimes, ${\boldsymbol V}_{\mathrm{A}}(t) = (v_1(t), v_2(t), \dots, v_k(t))$, where $k$ denotes the total number of customers (both $\CS$ and $\CA$) currently in $\mathrm{Q_A}$, and $v_i(t)$ denotes the residual triggering delay of the $i$-th customer at time $t$. The joint process $({\boldsymbol V}_{\mathrm{A}}(t), V_{\mathrm{M}}(t))$ is a continuous-time process defined on an appropriate general state space.

We work with a stationary and ergodic version of this joint process. Accordingly, the experience of a typical customer is characterized by the corresponding Palm distribution for its arrival class~\cite{baccelli2013elements}. Let $V_{\mathrm M}$ denote the stationary workload of $\mathrm{Q_M}$, and let $\boldsymbol V_{\mathrm A}$ denote the stationary vector of residual triggering delays in $\mathrm{Q_A}$.

We introduce two embedded versions of the workload process that differ only in their starting epochs. Specifically, consider a tagged $\CA$ and a tagged $\CP$--$\CS$ pair:
\begin{itemize}
    \item \textit{Reference process (RP)}:
    \[
    \{(\boldsymbol V^{\mathrm{RP}}_{\mathrm A}(t),V^{\mathrm{RP}}_{\mathrm M}(t)):t\ge0\},
    \]
    where $t=0$ is the arrival epoch of the tagged $\CA$ at $\mathrm{Q_A}$.
    
    \item \textit{Target process (TP)}:
    \[
    \{(\boldsymbol V^{\mathrm{TP}}_{\mathrm A}(t),V^{\mathrm{TP}}_{\mathrm M}(t)):t\ge0\},
    \]
    where $t=0$ is the epoch at which the tagged $\CP$  arrives at $\mathrm{Q_M}$ and its corresponding tagged $\CS$ simultaneously enters $\mathrm{Q_A}$.
\end{itemize}

We first characterize the initial states of these processes at time $t=0^+$.

\textit{Initial state of RP}: Since the arrival process of $\CA$ customers is a causally independent Poisson process, the PASTA property holds. Consequently, the tagged $\CA$, immediately after arriving at $\mathrm{Q_A}$ (i.e., at time $t=0^+$), sees $\mathrm{Q_M}$ in its steady-state workload $V_{\mathrm{M}}$. At the same instant, it observes $\mathrm{Q_A}$ in its stationary state, augmented by its own triggering delay $Y$. Thus,
\begin{equation} \label{eq:state_CA}
    \big(\boldsymbol{V}^{\mathrm{RP}}_{\mathrm{A}}(0^+), V^{\mathrm{RP}}_{\mathrm{M}}(0^+)\big)\overset{\mathrm{d}}{=} \big({\boldsymbol V}_{\mathrm{A}}\oplus Y,\ V_{\mathrm{M}}\big),
\end{equation}
where $\oplus$ denotes appending the new triggering delay $Y$.

\textit{Initial state of TP}: Upon the arrival of a $\CS$ at $\mathrm{Q_A}$, its corresponding $\CP$ arrives simultaneously at $\mathrm{Q_M}$. Let $(X^{\star}_{\CP}, X^{\star}_{\CS})\overset{\mathrm{d}}{=}(X_{\CP}, X_{\CS})$
denote the service requirements of the tagged pair.
Hence, the tagged $\CS$, immediately after arriving at $\mathrm{Q_A}$ (i.e., at time $t=0^+$), sees $\mathrm{Q_M}$ with workload $V_{\mathrm{M}} + X^{\star}_{\CP}$ --- i.e., the steady-state workload plus the service requirement of its corresponding tagged $\CP$, which arrived simultaneously at $\mathrm{Q_M}$. It sees $\mathrm{Q_A}$ in the same way as a $\CA$ arrival does, i.e., ${\boldsymbol V}_{\mathrm{A}}\oplus Y$. Thus,
\begin{equation} \label{eq:state_CS}
    \big(\boldsymbol{V}^{\mathrm{TP}}_{\mathrm{A}}(0^+), V^{\mathrm{TP}}_{\mathrm{M}}(0^+)\big)\overset{\mathrm{d}}{=} \big({\boldsymbol V}_{\mathrm{A}}\oplus Y,\ V_{\mathrm{M}} + X^{\star}_{\CP}\big).
\end{equation}

Observe that the (pre-arrival) steady-state workload of $\mathrm{Q_M}$ at time $t=0^-$, denoted by $V_{\mathrm M}$, is independent of the tagged service pair $(X^{\star}_{\CP},X^{\star}_{\CS})$. This follows because the tagged service requirements are assigned before arrival and are independent of the past trajectory of the workload process.

To make the analysis tractable, we \emph{couple} the RP and TP on a common probability space so that both processes experience the same random events from $t=0^+$ onward. Specifically, the coupling is constructed so that the two processes share the same realization of the initial states~\eqref{eq:state_CA} and~\eqref{eq:state_CS}, the same triggering delays, the same sequence of future $\CP$ arrivals and their corresponding $\CS$ arrivals, the same sequence of $\CA$ arrivals, and the same sequence of paired service requirements $\{(X_{\CP, i}, X_{\CS, i})\}_{i\geq 1}$. The only difference between the two processes at time $0^+$ is therefore the extra initial workload $X^\star_{\CP}$ present in $\mathrm{Q_M}$ of the TP, which yields
\begin{equation}
V^{\mathrm{TP}}_{\mathrm M}(0^+)=V^{\mathrm{RP}}_{\mathrm M}(0^+)+X^\star_{\CP}, \qquad \text{a.s.}
\end{equation}
Furthermore, both processes start from the identical initial configuration of $\mathrm{Q_A}$, namely ${\boldsymbol V}_{\mathrm A}\oplus Y$. We use this fact to show that the cumulative workload input received by $\mathrm{Q_M}$ during the triggering-delay interval $(0, Y)$ is identical in the RP and TP. 

\begin{lemma}\label{lem:ph:work_equiv}
Under the constructed coupling, the cumulative workload inputs to $\mathrm{Q_M}$ are almost surely (a.s.) identical for both RP and TP:
\begin{equation*}
A_{\mathrm{RP}}(t)=A_{\mathrm{TP}}(t)\coloneqq\bar A(t),
\qquad t \in (0,Y).
\end{equation*}
\end{lemma}
\begin{proof}
At $t=0^+$, the RP and TP differ only in the initial workload of $\mathrm{Q_M}$; their initial $\mathrm{Q_A}$ states are identical (${\boldsymbol V}_{\mathrm{A}}\oplus Y$). During the interval $(0, t]$, the cumulative work arriving at $\mathrm{Q_M}$ originates from two sources:
\begin{enumerate}
    \item \emph{Future $\CP$ arrivals}: By construction, both processes are subjected to the same Poisson arrival process and the same sequence of paired service times $\{(X_{\CP, i}, X_{\CS, i})\}_{i\ge1}$.
    \item \emph{Future $\CS$ arrivals}: These arrivals correspond to departures from $\mathrm{Q_A}$. Because $\mathrm{Q_A}$ is an infinite-server queue, its departure process depends solely on its state at $t=0^+$, namely $\boldsymbol{V}_{\mathrm{A}}\oplus Y$, and on subsequent arrivals to $\mathrm{Q_A}$, both of which are identical under the coupling.
\end{enumerate}
Since the realizations of all component processes that contribute new work to $\mathrm{Q_M}$ are identical under this coupling construction, it follows that $A_{\mathrm{RP}}(t)=A_{\mathrm{TP}}(t)\coloneqq\bar{A}(t)$ a.s.\ for all $t\in(0, Y)$.
\end{proof}

We now restrict attention to the workload processes of $\mathrm{Q_M}$, denoted by $V^{\mathrm{RP}}_{\mathrm M}(t)$ and $V^{\mathrm{TP}}_{\mathrm M}(t)$. From \eqref{eq:state_CA}--\eqref{eq:state_CS}, the initial workloads of $\mathrm{Q_M}$ satisfy
\[
V^{\mathrm{TP}}_{\mathrm M}(0^+)=V^{\mathrm{RP}}_{\mathrm M}(0^+)+X^\star_{\CP}, \qquad \text{a.s.}
\]
By Lemma~\ref{lem:ph:work_equiv}, the main queues $\mathrm{Q_M}$ in the RP and TP receive identical cumulative workload inputs throughout $(0,Y)$. Given their respective initial workloads and the common input process, we use the Skorokhod reflection mapping to track the workload trajectories  $V^{\mathrm{RP}}_{\mathrm M}(t)$ and $V^{\mathrm{TP}}_{\mathrm M}(t)$ from $t=0^+$ to $t=Y^-$. This sample-path comparison yields the relationship between $V^{\mathrm{TP}}_{\mathrm M}(Y^-)$ and $V^{\mathrm{RP}}_{\mathrm M}(Y^-)$. Here, $V^{\mathrm{TP}}_{\mathrm M}(Y^-)$ is precisely the workload in $\mathrm{Q_M}$ observed by the tagged $\CS$ immediately before it arrives at $\mathrm{Q_M}$, which determines its waiting time under the FCFS discipline.

\subsection{Sample-Path Workload Comparison via Skorokhod Reflection}
\label{subsec:ph:coupling}

For a work-conserving single-server queue, the workload process admits the well-known Skorokhod reflection representation~\cite{debicki2015queues}. In particular, given an initial workload $V(0)$ and a cumulative workload input $A(t)$, the workload at time $t$ is given by
\begin{equation}
    V(t)=V(0)+A(t)-t+
    \max\left\{0,-\inf_{0\leq s\leq t}
    \left(V(0)+A(s)-s\right)\right\}.
\end{equation}

Applying this reflection mapping to our coupled processes at time $t=Y^-$, and using the common input $\bar{A}(\cdot)$ established in Lemma~\ref{lem:ph:work_equiv}, we obtain for the RP that
\begin{equation}\label{eq:ph:ref_work}
    V^{\mathrm{RP}}_{\mathrm{M}}(Y^-) = V_{\mathrm{M}} +  \bar{A}(Y^-) - Y + \max\Bigl\{0, -\inf_{s \in (0,Y)}\big(V_{\mathrm{M}} +  \bar{A}(s) - s\big)\Bigr\},
\end{equation}
and for the TP
\begin{equation}\label{eq:ph:target_work}
        V^{\mathrm{TP}}_{\mathrm{M}}(Y^-) = V_{\mathrm{M}} + X^{\star}_{\CP} +  \bar{A}(Y^-) - Y +\max\Bigl\{0, -\inf_{s \in (0,Y)}\big(V_{\mathrm{M}} + X^{\star}_{\CP} +  \bar{A}(s) - s\big)\Bigr\}.
\end{equation}

Let $I(Y^-)$ denote the infimum of the unreflected workload process of $\mathrm{Q_M}$ in the TP over the interval $(0, Y)$:
\begin{equation*}
    I(Y^-) \coloneqq \inf_{s\in(0,Y)}\big(V_{\mathrm{M}} + X^{\star}_{\CP} +  \bar{A}(s) - s\big).
\end{equation*}
The corresponding infimum in the RP is $I(Y^-)-X^\star_{\CP}$, because the two unreflected paths differ by the constant $X^\star_{\CP}$. Substituting these infima into~\eqref{eq:ph:ref_work} and~\eqref{eq:ph:target_work}, the workload difference at time $Y^-$ becomes
\begin{align}
    \Delta
    \coloneqq& V^{\mathrm{TP}}_{\mathrm{M}}(Y^-) - V^{\mathrm{RP}}_{\mathrm{M}}(Y^-)\nonumber \\ 
    =& X^{\star}_{\CP} + \max\{0, -I(Y^-)\} - \max\{0, X^{\star}_{\CP}-I(Y^-)\} \nonumber \\ 
    =& \min\bigl\{ X^{\star}_{\CP}, \max(0, I(Y^-)) \bigr\}. \label{eq:delta_identity}
\end{align}
The final equality follows by considering the three possible regimes for $I(Y^-)$. Consequently,
\begin{equation}
    V^{\mathrm{TP}}_{\mathrm{M}}(Y^-)= V^{\mathrm{RP}}_{\mathrm{M}}(Y^-) + \Delta ,
\end{equation}
where
\begin{equation}\label{eq:ph:workload_relation}
    \Delta = 
    \begin{cases}
    0, & \text{if } I(Y^-) \le 0,\\
    I(Y^-), & \text{if } 0 < I(Y^-) < X^{\star}_{\CP},\\
    X^{\star}_{\CP}, & \text{if } I(Y^-) \ge X^{\star}_{\CP}.
    \end{cases}
\end{equation}
In particular, since $0 \le \Delta \le X^{\star}_{\CP}$ pathwise,
\begin{equation}\label{eq:bounds}
    V^{\mathrm{RP}}_{\mathrm{M}}(Y^-)\leq
    V^{\mathrm{TP}}_{\mathrm{M}}(Y^-)\leq
    V^{\mathrm{RP}}_{\mathrm{M}}(Y^-)+X^{\star}_{\CP},
    \qquad\text{a.s.}
\end{equation}

Thus, under the coupling, the workload in $\mathrm{Q_M}$ seen by the tagged $\CS$ arrival is always at least as large as that in the RP, and the excess is bounded above by the service requirement of the corresponding tagged $\CP$.

\begin{figure*}
    \centering
\begin{tikzpicture}
    \node[inner sep=8pt] (legend) {
        \begin{tabular}{rlrl}
        \tikz\draw[blue!90!black, line width=1.5pt] (0,0) -- (0.6,0); &\hspace{-0.5cm} \footnotesize $V^{\mathrm{RP}}_{\mathrm{M}}(t)$ \quad\quad\quad &
        \tikz\draw[red!80!black, line width=1.5pt, dashed] (0,0) -- (0.6,0); &\hspace{-0.5cm} \footnotesize  $V^{\mathrm{TP}}_{\mathrm{M}}(t)$ 
        \end{tabular}
    };
\end{tikzpicture}
\vspace{-0.25cm}

\subfloat[Fully Served ($I(Y^-) \le 0$).\label{fig:ph:case1}]{%
    \resizebox{0.325\textwidth}{!}{%
    \begin{tikzpicture}[font=\large]
        \useasboundingbox (-1.0, -0.3) rectangle (10.8, 6.5);
        
        \draw[help lines, gray, dotted, thick] (0,0) grid (10,6);
        
        \fill[red!10]
            (0,2) -- (1,1) -- (1,3) -- (3.5,0.5) -- (3.5,1.5) -- (5,0) -- (6.5,0) -- (6.5,1) -- (7.5,0) -- (8.5,0) --
            (8.5,0) -- (6.5,2) -- (6.5,1) -- (5,2.5) -- (3.5,4) -- (3.5,3) -- (1,5.5) -- (1,3.5) -- (0,4.5) -- cycle;

        \draw[->, line width=2pt] (-0.2,0) -- (10.5,0) node[right] {$t$};
        \draw[->, line width=2pt] (0,-0.2) -- (0,6.3) node[above] {$V_{\mathrm{M}}(t)$};
        \node[below left] at (0,0) {$0$};
        
        \draw[densely dotted, line width=1.7pt, gray!80] (9.5, 0) -- (9.5, 6.0);
        \node[below] at (9.5, -0.1) {$Y$};
        
        \draw[line width=1.2pt] (-0.15, 2) -- (0.15, 2) node[left=4pt] {$V_{\mathrm{M}}$};
        \draw[line width=1.2pt] (-0.15, 4.5) -- (0.15, 4.5) node[left=4pt] {$V_{\mathrm{M}} \!+\! X^{\star}_{\CP}$};

        \draw[blue!90!black, line width=2pt] 
            (0,2) -- (1,1) -- (1,3) -- (3.5,0.5) -- (3.5,1.5) -- (5,0) -- (6.5,0) -- (6.5,1) -- (7.5,0) -- (9,0) -- (9,1.5) -- (9.5,1);
            
        \draw[red!80!black, line width=2pt, dashed] 
            (0,4.5) -- (1,3.5) -- (1,5.5) -- (3.5,3) -- (3.5,4) -- (5,2.5) -- (6.5,1) -- (6.5,2) -- (8.5,0) -- (9,0) -- (9,1.5) -- (9.5,1);

        \draw[<->, >=stealth, line width=1.5pt] (0.35, 2.05) -- (0.35, 4.45) 
            node[midway, right, fill=red!10, inner sep=1.5pt, rounded corners=1pt] {$X^{\star}_{\CP}$};

    \end{tikzpicture}%
    }%
}\hfill
\subfloat[Partially Served ($0 < I(Y^-) < X^{\star}_{\CP}$).\label{fig:ph:case2}]{%
    \resizebox{0.325\textwidth}{!}{%
    \begin{tikzpicture}[font=\large]
        \useasboundingbox (-1.0, -0.3) rectangle (10.8, 6.5);
        
        \draw[help lines, gray, dotted, thick] (0,0) grid (10,6);
        
        \fill[red!10]
            (0,2) -- (1,1) -- (1,3) -- (3.5,0.5) -- (3.5,1.5) -- (5,0) -- (6,0) -- (6,1) -- (7,0) -- (8,0) -- (8,2) -- (9.5,0.5) --
            (9.5,1) -- (8,2.5) -- (8,0.5) -- (7,1.5) -- (6,2.5) -- (6,1.5) -- (5,2.5) -- (3.5,4) -- (3.5,3) -- (1,5.5) -- (1,3.5) -- (0,4.5) -- cycle;

        \draw[->, line width=2pt] (-0.2,0) -- (10.5,0) node[right] {$t$};
        \draw[->, line width=2pt] (0,-0.2) -- (0,6.3) node[above] {$V_{\mathrm{M}}(t)$};
        \node[below left] at (0,0) {$0$};
        
        \draw[densely dotted, line width=1.7pt, gray!80] (9.5, 0) -- (9.5, 6.0);
        \node[below, ] at (9.5, -0.1) {$Y$};
        
        \draw[line width=1.2pt] (-0.15, 2) -- (0.15, 2);
        \draw[line width=1.2pt] (-0.15, 4.5) -- (0.15, 4.5);

        \draw[blue!90!black, line width=2pt] 
            (0,2) -- (1,1) -- (1,3) -- (3.5,0.5) -- (3.5,1.5) -- (5,0) -- (6,0) -- (6,1) -- (7,0) -- (8,0) -- (8,2) -- (9.5,0.5);
            
        \draw[red!80!black, line width=2pt, dashed] 
            (0,4.5) -- (1,3.5) -- (1,5.5) -- (3.5,3) -- (3.5,4) -- (5,2.5) -- (6,1.5) -- (6,2.5) -- (7,1.5) -- (8,0.5) -- (8,2.5) -- (9.5,1);

        \draw[<->, >=stealth, line width=1.5pt] (0.35, 2.05) -- (0.35, 4.45) 
            node[midway, right, fill=red!10, inner sep=1.5pt, rounded corners=1pt] {$X^{\star}_{\CP}$};
        
        \draw[<->, >=stealth, line width=1.5pt] (9.7, 0.5) -- (9.7, 1);
        \node[right] at (9.7, 0.75) {$I(Y^-)$};

    \end{tikzpicture}%
    }%
}\hfill
\subfloat[Unserved ($I(Y^-) \ge X^{\star}_{\CP}$).\label{fig:ph:case3}]{%
    \resizebox{0.325\textwidth}{!}{%
    \begin{tikzpicture}[font=\large]
        \useasboundingbox (-1.0, -0.3) rectangle (10.8, 6.5);
        
        \draw[help lines, gray, dotted, thick] (0,0) grid (10,6);
        
        \fill[red!10]
            (0,2) -- (1,1) -- (1,3) -- (3.5,0.5) -- (3.5,2.5) -- (5,1) -- (5,2) -- (6,1) -- (6,2) -- (7.5,0.5) -- (7.5,2.5) -- (9.5,0.5) --
            (9.5,3) -- (7.5,5) -- (7.5,3) -- (6,4.5) -- (6,3.5) -- (5,4.5) -- (5,3.5) -- (3.5,5) -- (3.5,3) -- (1,5.5) -- (1,3.5) -- (0,4.5) -- cycle;

        \draw[->, line width=2pt] (-0.2,0) -- (10.5,0) node[right] {$t$};
        \draw[->, line width=2pt] (0,-0.2) -- (0,6.3) node[above] {$V_{\mathrm{M}}(t)$};
        \node[below left] at (0,0) {$0$};
        
        \draw[densely dotted, line width=1.7pt, gray!80] (9.5, 0) -- (9.5, 6.0);
        \node[below, ] at (9.5, -0.1) {$Y$};
        
        \draw[line width=1.2pt] (-0.15, 2) -- (0.15, 2);
        \draw[line width=1.2pt] (-0.15, 4.5) -- (0.15, 4.5);

        \draw[blue!90!black, line width=2pt] 
            (0,2) -- (1,1) -- (1,3) -- (3.5,0.5) -- (3.5,2.5) -- (5,1) -- (5,2) -- (6,1) -- (6,2) -- (7.5,0.5) -- (7.5,2.5) -- (9.5,0.5);
            
        \draw[red!80!black, line width=2pt, dashed] 
            (0,4.5) -- (1,3.5) -- (1,5.5) -- (3.5,3) -- (3.5,5) -- (5,3.5) -- (5,4.5) -- (6,3.5) -- (6,4.5) -- (7.5,3) -- (7.5,5) -- (9.5,3);

        \draw[<->, >=stealth, line width=1.5pt] (0.35, 2.05) -- (0.35, 4.45) 
            node[midway, right, inner sep=1.5pt, rounded corners=1pt] {$X^{\star}_{\CP}$};
        
        \draw[<->, >=stealth, line width=1.5pt] (9.7, 0.55) -- (9.7, 2.95);
        \node[right] at (9.7, 1.75) {$X^{\star}_{\CP}$};
    \end{tikzpicture}%
    }%
}
    \caption{Sample-path coupling over the triggering interval $Y$, illustrating the three regimes in Theorem~\ref{thm:ph:conditional_dist}. The RP and TP evolve under identical realizations of all future primitives.}
    \label{fig:ph:coupling_cases}
\end{figure*}
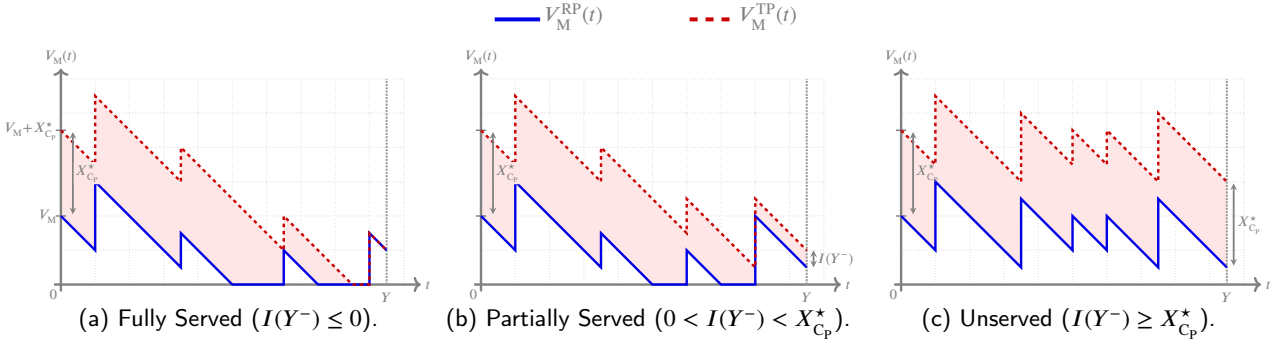

Since the workload of a work-conserving single-server queue is invariant across all work-conserving service disciplines~\cite[Ch.~3.4, p.~113]{Kleinrock1976Volume2}, the structure of \eqref{eq:ph:workload_relation} is best understood by considering a work-conserving discipline in which the tagged $\CP$ has the lowest priority relative to all other customers (both the pre-existing arrivals and all arrivals during $(0, Y)$). Under this discipline, the three cases of \eqref{eq:ph:workload_relation} describe the degree of erosion of the initial workload difference $X^{\star}_{\CP}$:

\begin{itemize}
    \item \emph{Fully served} ($I(Y^-)\le 0$): The extra work contributed by the tagged $\CP$ is fully exhausted by $t=Y^-$ (with the queue emptying at least once during $(0, Y)$; see Fig.~\ref{fig:ph:case1}). Hence, $V^{\mathrm{TP}}_{\mathrm M}(Y^-)=V^{\mathrm{RP}}_{\mathrm M}(Y^-)$.

    \item \emph{Partially served} ($0<I(Y^-)<X^{\star}_{\CP}$): The tagged $\CP$ receives partial service, so only part of the initial excess workload remains (see Fig.~\ref{fig:ph:case2}). In this case, $V^{\mathrm{TP}}_{\mathrm M}(Y^-)=V^{\mathrm{RP}}_{\mathrm M}(Y^-)+I(Y^-)$.

    \item \emph{Unserved} ($I(Y^-)\ge X^{\star}_{\CP}$): The tagged $\CP$ receives no service before $t=Y^-$, and higher-priority work occupies the server throughout $(0,Y)$ (see Fig.~\ref{fig:ph:case3}). Therefore, $V^{\mathrm{TP}}_{\mathrm M}(Y^-)=V^{\mathrm{RP}}_{\mathrm M}(Y^-)+X^{\star}_{\CP}$.
\end{itemize}

Finally, note that $V^{\mathrm{RP}}_{\mathrm M}(Y^-)$ and $I(Y^-)$ are generally dependent, because both are determined by the same common input path over $(0, Y)$. By contrast, $V^{\mathrm{RP}}_{\mathrm M}(Y^-)$ is independent of the tagged pair $(X^\star_{\CP}, X^\star_{\CS})$, since no tagged primary customer is present in the RP. This distinction will be important when we pass from pathwise bounds to distributional bounds.

\subsection{Distributional Identification and Stochastic Ordering}
\label{subsec:ph:dep_proc}

The previous subsection established a pathwise relation between the workload of $\mathrm{Q_M}$ in RP and TP at the end of the triggering delay interval. We now convert this relation into a distributional statement. The key step is to relate the RP observation at time $t=Y^-$ to the steady-state workload distribution of $\mathrm{Q_M}$.

\begin{lemma}\label{lem:ph:poisson_dep}
In the $\mathrm{M/G/\infty}$ auxiliary queue $\mathrm{Q_A}$, the departure process of $\CA$ customers is a \emph{causally independent} Poisson process.
\end{lemma}
\begin{proof}
 See Appendix~\ref{app:poisson_dep}
\end{proof}

By Lemma~\ref{lem:ph:poisson_dep}, the departure process of $\CA$ customers from $\mathrm{Q_A}$ is a \emph{causally independent} Poisson process. Therefore, by PASTA, the tagged $\CA$ sees $\mathrm{Q_M}$ with steady-state workload $V_{\mathrm{M}}$ at time $Y^-$: $V^{\mathrm{RP}}_{\mathrm{M}}(Y^-) \overset{\mathrm{d}}{=} V_{\mathrm{M}}$.

By contrast, the departure process of $\CS$ customers from $\mathrm{Q_A}$ (equivalently, their arrival process at $\mathrm{Q_M}$) is \emph{not} independent of the workload process of $\mathrm{Q_M}$. Let $W_{\CS}$ denote the workload seen by an arriving $\CS$ at $\mathrm{Q_M}$, so that $V^{\mathrm{TP}}_{\mathrm{M}}(Y^-) \overset{\mathrm{d}}{=} W_{\CS}$.

Combining these observations with \eqref{eq:ph:workload_relation} gives an exact structural representation of the workload seen by a $\CS$. It is not yet a closed-form distribution, because the two terms in the representation are dependent.

\begin{theorem}\label{thm:ph:conditional_dist}
In the $\mathrm{M^T/G/1}$ queue, the workload seen by an arriving $\CS$, denoted by $W_{\CS}$, satisfies the distributional relationship
\begin{equation}
    W_{\CS} \overset{\mathrm{d}}{=} V^{\mathrm{RP}}_{\mathrm M}(Y^-) + \Delta,
\end{equation}
where $V^{\mathrm{RP}}_{\mathrm M}(Y^-)\overset{\mathrm{d}}{=}V_{\mathrm M}$ and $\Delta = \min\bigl\{ X^{\star}_{\CP}, \max\{0, I(Y^-)\} \bigr\}$.
\end{theorem}

The key structural insight is that $V^{\mathrm{RP}}_{\mathrm M}(Y^-)$ is \emph{not} independent of $\Delta$. This dependence is natural: for instance, when $0<I(Y^-)<X^{\star}_{\CP}$, a small value of $I(Y^-)$ indicates that the unreflected workload trajectory remained close to zero over $(0, Y)$, which is associated with a smaller $V^{\mathrm{RP}}_{\mathrm M}(Y^-)$. Knowing their separate marginal distributions is therefore insufficient to evaluate the sum. The pathwise bounds $0\le\Delta\le X^\star_{\CP}$ avoid this difficulty and imply the following stochastic order.

\begin{corollary}\label{cor:ordering}
The workload seen by an arriving $\CS$ satisfies the stochastic ordering
\begin{equation*}
    V_{\mathrm{M}} \le_{\mathrm{st}} W_{\CS} \le_{\mathrm{st}} V_{\mathrm{M}} + X^{\star}_{\CP},
\end{equation*}
where $X^{\star}_{\CP}\overset{\mathrm{d}}{=}X_{\CP}$ and $X^{\star}_{\CP}$ is independent of $V_{\mathrm{M}}$.
\end{corollary}

The coupling compares a secondary arrival's workload with the stationary workload without assuming PASTA for secondary customers. The next section inserts these pathwise bounds into a level-crossing identity to obtain computable workload and system-time bounds.

\section{Workload and System-Time Distribution Bounds}
\label{sec:bounds}

This section establishes lower and upper bounds, in the usual stochastic order, on the stationary workload distribution and the class-specific system-time distributions in the $\mathrm{M^T/G/1}$ queue. We first use level-crossing arguments to derive an exact identity for the tail distribution of the stationary workload, expressing it as the sum of two contributions corresponding to $\CP$ and $\CS$ arrivals. We evaluate the contribution of $\CP$ arrivals exactly using the PASTA property, while bounding that of $\CS$ arrivals using the sample-path coupling established in Section~\ref{sec:ph}. Combining these results yields the stochastic sandwich in \eqref{eq:workload_bounds}. We then show that these bounds are asymptotically tight in the sense that the stationary workload converges in distribution to that of the lower-bound queue as the triggering delays diverge in probability to infinity, and to that of the upper-bound queue as the triggering delays converge in probability to zero. Finally, we translate these results into stochastic bounds on the class-specific system-time distributions and present a practical heuristic approximation.

\subsection{The Stationary Workload Tail Distribution}
\label{subsec:exact_workload_identity}
We begin with an exact workload identity, which will serve as the common starting point for deriving both the lower and upper bounds. Define 
\[
    \overline{B}_{\CP + \CS}(x) \coloneqq \Prob(X_{\CP} + X_{\CS} > x),
\]
where the sum is evaluated under the joint distribution of the service pair $(X_{\CP}, X_{\CS})$. Let $V_{\mathrm{M}}$ denote the stationary workload of the $\mathrm{M^T/G/1}$ queue, and define
\[
    \overline{F}_{V_{\mathrm{M}}}(x) \coloneqq \Prob(V_{\mathrm{M}} > x), \qquad x \in \mathbb{R}.
\]
As usual, we extend the tail function by setting $\overline{F}_{V_{\mathrm{M}}}(x) = 1$ for $x < 0$.

For $k \in \{\CP, \CS\}$, let $W_k$ be the workload seen by a tagged class-$k$ customer immediately before its arrival, and let $X_k^\star$ be its service requirement. Such a customer causes an up-crossing of a level $y > 0$ if $\{ W_k \le y < W_k + X_k^\star \}$. Since both $\CP$ and $\CS$ customers arrive at long-run rate $\lambda$, the total up-crossing rate at level $y$ is
\begin{equation}
\label{eq:upcrossing_rate}
    \sum_{k \in \{\CP, \CS\}} \lambda \Prob(W_k \le y < W_k + X_k^\star).
\end{equation}

An arrival can jump the workload upward across $y$, whereas service brings it downward continuously at a rate of one. In stationarity, the rates of upward and downward crossings are equal~\cite{brill2017level}. Because the workload decreases at a unit rate whenever it is positive, the down-crossing rate at level $y$ is equal to the density of the absolutely continuous component of the stationary workload distribution at $y$, denoted by $f_{V_{\mathrm{M}}}(y)$. Therefore, for almost every $y > 0$,
\begin{equation}
\label{eq:density_balance}
    f_{V_{\mathrm{M}}}(y) 
    = \sum_{k \in \{\CP, \CS\}} \lambda \Prob( W_k \le y < W_k + X_k^\star ),
    \qquad y > 0.
\end{equation}
Here we use the standard fact that the stationary workload distribution of a stable single-server queue may have an atom at the origin (with $\Prob(V_{\mathrm{M}} = 0) = 1 - \rho$) but is absolutely continuous on $(0, \infty)$. This is sufficient for the integral manipulations below.

Fix $x \ge 0$ and integrate \eqref{eq:density_balance} over $[x, \infty)$: 
\begin{align}
\label{eq:density_integrated_step_by_step}
    \overline{F}_{V_{\mathrm{M}}}(x)
    &= \int_x^\infty f_{V_{\mathrm{M}}}(y) \diff y \nonumber \\
    &= \int_x^\infty \sum_{k \in \{\CP, \CS\}} \lambda \Prob( W_k \le y < W_k + X_k^\star ) \diff y \nonumber \\
    &= \sum_{k \in \{\CP, \CS\}} \lambda \int_x^\infty \Prob( W_k \le y < W_k + X_k^\star ) \diff y \nonumber \\
    &= \sum_{k \in \{\CP, \CS\}} \lambda \int_x^\infty \E\left[ \mathbbm{1}_{\{ W_k \le y < W_k + X_k^\star \}} \right] \diff y \nonumber \\
    &= \sum_{k \in \{\CP, \CS\}} \lambda \E\left[ \int_x^\infty \mathbbm{1}_{\{ W_k \le y < W_k + X_k^\star \}} \diff y \right].
\end{align}
Since the integrands are non-negative, Tonelli's theorem justifies the interchange of expectation and integration.
We shall apply Tonelli's theorem repeatedly throughout the subsequent derivations without further explicit mention.

For any non-negative $w$ and $b$, substituting $t = y - w$ yields
\begin{align}
    \int_x^\infty \mathbbm{1}_{\{ w \le y < w + b \}} \diff y
    &= \int_{x - w}^\infty \mathbbm{1}_{\{ 0 \le t < b \}} \diff t 
       && \text{(substitute $t = y - w$)} \nonumber \\
    &= \int_0^b \mathbbm{1}_{\{ w + t > x \}} \diff t.
       && \text{(rewrite the integration domain)} \label{eq:indicator_final}
\end{align}
Therefore,
\begin{equation}
\label{eq:palm_tail}
    \overline{F}_{V_{\mathrm{M}}}(x)
    = \lambda \E\left[ \int_0^{X_{\CP}^\star} \mathbbm{1}_{\{ W_{\CP} + t > x \}} \diff t \right]
    + \lambda \E\left[ \int_0^{X_{\CS}^\star} \mathbbm{1}_{\{ W_{\CS} + t > x \}} \diff t \right],
    \qquad x \ge 0.
\end{equation}
This equation is exact and serves as the starting point for all subsequent bounds. Its two terms have fundamentally different analytical properties. On the right-hand side, the first expectation represents the contribution of primary customers in $\overline{F}_{V_{\mathrm{M}}}(x)$, while the second captures the contribution of secondary customers. Simplifying the first expectation is straightforward because primary arrivals form a causally independent Poisson process, so the PASTA property holds. As demonstrated in Subsection~\ref{subsec:primary_contribution}, this allows the \textit{exact} contribution of primary customers to be expressed as a standard convolution integral. In contrast, evaluating the \textit{exact} contribution of secondary customers is not tractable in closed form, because PASTA does not directly apply to secondary arrivals: the secondary arrival process fails the lack-of-anticipation assumption (i.e., it is not causally independent, in the sense of Definition~\ref{def:exp:causal_indep}). Therefore, Subsection~\ref{subsec:bounding_secondary} establishes explicit lower and upper bounds for the contribution of secondary customers.

\subsection{Contribution of Primary Customers}
\label{subsec:primary_contribution}

We first handle the simpler term in~\eqref{eq:palm_tail}. Since $\CP$ arrivals form a causally independent Poisson process, the workload seen by a tagged primary customer has the stationary distribution:
\begin{equation}
\label{eq:primary_pasta}
    W_{\CP} \overset{\mathrm{d}}{=} V_{\mathrm{M}}.
\end{equation}
Moreover, each $\CP$--$\CS$ pair is assigned an i.i.d. service pair $(X_{\CP}^\star, X_{\CS}^\star) \overset{\mathrm{d}}{=} (X_{\CP}, X_{\CS})$, independent of the arrival process and of the pre-arrival system state. Hence, $W_{\CP}$ is independent of the tagged service pair $(X_{\CP}^\star, X_{\CS}^\star)$, and in particular of $X_{\CP}^\star$.

Using \eqref{eq:primary_pasta} and the independence between $W_{\CP}$ and $X_{\CP}^\star$, the primary contribution in \eqref{eq:palm_tail} becomes
\begin{align}
    \lambda \E\left[ \int_0^{X_{\CP}^\star} \mathbbm{1}_{\{ W_{\CP} + t > x \}} \diff t \right] \nonumber
    &= \lambda \int_0^\infty \E\left[ \mathbbm{1}_{\{ t < X_{\CP}^\star \}} \mathbbm{1}_{\{ W_{\CP} > x - t \}} \right] \diff t \nonumber \\
    &= \lambda \int_0^\infty \Prob( t < X_{\CP}^\star, W_{\CP} > x - t ) \diff t \nonumber \\
    &= \lambda \int_0^\infty \Prob( X_{\CP}^\star > t ) \Prob( W_{\CP} > x - t ) \diff t 
        && \text{(by independence)} \nonumber \\
    &= \lambda \int_0^\infty \overline{B}_{\CP}(t) \overline{F}_{V_{\mathrm{M}}}(x - t) \diff t.
    \label{eq:primary_conv_raw_1}
\end{align}

\subsection{Bounds on Contribution of Secondary Customers}
\label{subsec:bounding_secondary}

Unlike the $\CP$ arrivals, the $\CS$ arrival process is not independent of the workload process of $\mathrm{Q_M}$. Each $\CS$ arrival is causally linked to its corresponding $\CP$ arrival, which has already contributed to the queue workload. Consequently, $\CS$ arrival epochs generally depend on the workload process, meaning that $W_{\CS}$ is generally neither distributed as $V_{\mathrm{M}}$ nor independent of the associated service pair $(X_{\CP}^\star, X_{\CS}^\star)$.

Because the workload $W_{\CS}$ seen by a secondary customer is not independent of $(X_{\CP}^\star, X_{\CS}^\star)$, obtaining an exact representation of the secondary contribution in 
\eqref{eq:palm_tail} is challenging. We therefore return to the joint sample-path coupling underlying Corollary~\ref{cor:ordering}. This is essential because a marginal stochastic order alone would not justify inserting a possibly dependent tagged service requirement into the integral.

To this end, we exploit the coupling construction of Section~\ref{subsec:ph:coupling}. Consider a tagged $\CS$ under this coupling. Let $V_{\mathrm{M}}^{\mathrm{TP}}(Y^-)$ denote the workload seen by the tagged $\CS$ in the TP, and let $V_{\mathrm{M}}^{\mathrm{RP}}(Y^-)$ denote the workload in the RP at the same epoch. The coupling relation \eqref{eq:bounds} yields 
\begin{equation}
    \label{eq:almost_sure_bound}
    V_{\mathrm{M}}^{\mathrm{RP}}(Y^-) \le V_{\mathrm{M}}^{\mathrm{TP}}(Y^-) \le V_{\mathrm{M}}^{\mathrm{RP}}(Y^-) + X_{\CP}^\star, \qquad \text{a.s.}
\end{equation}

The reference workload $V_{\mathrm{M}}^{\mathrm{RP}}(Y^-)$ is independent of the tagged service pair $(X_{\CP}^\star, X_{\CS}^\star)$ (as established in Section~\ref{sec:ph}). Since the mapping
\[
    w \mapsto \int_0^{X_{\CS}^\star} \mathbbm{1}_{\{ w + t > x \}} \diff t
\]
is non-decreasing in $w$ for each fixed $x$ and $X_{\CS}^\star$, the inequalities in \eqref{eq:almost_sure_bound} imply the following pathwise inequalities:
\begin{equation}
\label{eq:secondary_integral_pathwise}
    \int_0^{X_{\CS}^\star} \mathbbm{1}_{\{ V_{\mathrm{M}}^{\mathrm{RP}}(Y^-) + t > x \}} \diff t
    \le \int_0^{X_{\CS}^\star} \mathbbm{1}_{\{ V_{\mathrm{M}}^{\mathrm{TP}}(Y^-) + t > x \}} \diff t
    \le \int_0^{X_{\CS}^\star} \mathbbm{1}_{\{ V_{\mathrm{M}}^{\mathrm{RP}}(Y^-) + X_{\CP}^\star + t > x \}} \diff t.
\end{equation}

Recall that
\[
    V_{\mathrm{M}}^{\mathrm{TP}}(Y^-) \overset{\mathrm{d}}{=}  W_{\CS},
    \qquad
    V_{\mathrm{M}}^{\mathrm{RP}}(Y^-) \overset{\mathrm{d}}{=} V_{\mathrm{M}}.
\]
From our coupling construction, we therefore obtain the following stochastic orderings:
\begin{equation}
\label{eq:secondary_integral_stochastic_bound}
    \int_0^{X_{\CS}^\star} \mathbbm{1}_{\{ V_{\mathrm{M}} + t > x \}} \diff t
    \le_{\mathrm{st}} \int_0^{X_{\CS}^\star} \mathbbm{1}_{\{ W_{\CS} + t > x \}} \diff t
    \le_{\mathrm{st}} \int_0^{X_{\CS}^\star} \mathbbm{1}_{\{ V_{\mathrm{M}} + X_{\CP}^\star + t > x \}} \diff t,
\end{equation}
where in both bounding terms, $V_{\mathrm{M}}$ is independent of the tagged service pair $(X_{\CP}^\star, X_{\CS}^\star)$. Since stochastic ordering implies ordering of expectations for the non-negative random variables, we obtain
\begin{equation}
\label{eq:secondary_integral_expectation_bound}
    \E\left[ \int_0^{X_{\CS}^\star} \mathbbm{1}_{\{ V_{\mathrm{M}} + t > x \}} \diff t \right]
    \le \E\left[ \int_0^{X_{\CS}^\star} \mathbbm{1}_{\{ W_{\CS} + t > x \}} \diff t \right]
    \le \E\left[ \int_0^{X_{\CS}^\star} \mathbbm{1}_{\{ V_{\mathrm{M}} + X_{\CP}^\star + t > x \}} \diff t \right].
\end{equation}
Because $V_{\mathrm{M}}$ and $X_{\CS}^\star$ are independent, the lower bound in \eqref{eq:secondary_integral_expectation_bound} simplifies as in \eqref{eq:primary_conv_raw_1}:
\begin{equation}
    \E\left[ \int_0^{X_{\CS}^\star} \mathbbm{1}_{\{ V_{\mathrm{M}} + t > x \}} \diff t \right]
    = \int_0^\infty \overline{B}_{\CS}(t) \overline{F}_{V_{\mathrm{M}}}(x - t) \diff t.
    \label{eq:secondary_lower_conv}
\end{equation}
For the upper bound, substituting $t$ by $t-X_{\CP}^\star$ inside the integral yields
\begin{equation}
\label{eq:secondary_upper_conditional}
    \int_0^{X_{\CS}^\star} \mathbbm{1}_{\{ V_{\mathrm{M}} + X_{\CP}^\star + t > x \}} \diff t
    \implies \int_{X_{\CP}^\star}^{X_{\CP}^\star + X_{\CS}^\star} \mathbbm{1}_{\{ V_{\mathrm{M}} + t > x \}} \diff t.
\end{equation}
Combining \eqref{eq:secondary_lower_conv} and \eqref{eq:secondary_upper_conditional} with \eqref{eq:secondary_integral_expectation_bound}, we arrive at
\begin{equation}
\label{eq:integral_bounds_explicit}
    \int_0^\infty \overline{B}_{\CS}(t) \overline{F}_{V_{\mathrm{M}}}(x - t) \diff t
    \le \E\left[ \int_0^{X_{\CS}^\star} \mathbbm{1}_{\{ W_{\CS} + t > x \}} \diff t \right]
    \le \E\left[ \int_{X_{\CP}^\star}^{X_{\CP}^\star + X_{\CS}^\star}  \mathbbm{1}_{\{ V_{\mathrm{M}} + t > x \}} \diff t \right].
\end{equation}

\subsection{Workload Stochastic Bounds}
\label{subsec:operators}

In this subsection, we combine the previous bounds with the exact representation of the primary contribution to derive lower and upper bounds for the stationary workload of the queue.

For the lower bound, we substitute the exact primary contribution from \eqref{eq:primary_conv_raw_1} and the lower bound for the secondary contribution from \eqref{eq:integral_bounds_explicit} into the workload identity \eqref{eq:palm_tail}. Thus, for every $x \ge 0$,
\begin{equation}
\label{eq:lower_bound}
    \overline{F}_{V_{\mathrm{M}}}(x)
    \ge \lambda \int_0^\infty \overline{B}_{\CP}(t) \overline{F}_{V_{\mathrm{M}}}(x - t) \diff t
    + \lambda \int_0^\infty \overline{B}_{\CS}(t) \overline{F}_{V_{\mathrm{M}}}(x - t) \diff t.
\end{equation}

For the upper bound, substituting the upper inequality of \eqref{eq:integral_bounds_explicit} into \eqref{eq:palm_tail} gives
\begin{equation}
\label{eq:upper_bound_expectation_form}
    \overline{F}_{V_{\mathrm{M}}}(x)
    \le \lambda \E\left[ \int_0^{X_{\CP}^\star} \mathbbm{1}_{\{ W_{\CP} + t > x \}} \diff t \right]
    + \lambda \E\left[ \int_{X_{\CP}^\star}^{X_{\CP}^\star + X_{\CS}^\star}  \mathbbm{1}_{\{ V_{\mathrm{M}} + t > x \}} \diff t \right].
\end{equation}
Because $W_{\CP} \overset{\mathrm{d}}{=} V_{\mathrm{M}}$ and $W_{\CP}$ is independent of $(X_{\CP}^\star, X_{\CS}^\star)$, we obtain
\begin{equation}
\label{eq:upper_bound}
    \overline{F}_{V_{\mathrm{M}}}(x)
    \le \lambda \E\left[ \int_0^{X_{\CP}^\star + X_{\CS}^\star} \mathbbm{1}_{\{ V_{\mathrm{M}} + t > x \}} \diff t \right] = \lambda \int_0^\infty \overline{B}_{\CP + \CS}(t) \overline{F}_{V_{\mathrm{M}}}(x - t) \diff t.
\end{equation}

The inequalities in \eqref{eq:lower_bound} and \eqref{eq:upper_bound} are not yet directly applicable for performance evaluation, as they contain the unknown tail distribution function $\overline{F}_{V_{\mathrm{M}}}$ on both sides. To transform these relations into computable bounds, we recast these inequalities as functional mappings. In the following subsection, we formalize this approach and the resulting stochastic ordering properties by introducing the corresponding workload operators and their fixed points.

\subsection{Fixed-Point Operators}

Define the space of admissible workload tail functions by
\[
    \overline{\mathcal F}
    =
    \left\{
      \overline{F}:\mathbb{R}\to[0,1]:
      \begin{array}{l}
        \overline{F}(x)=1 \quad \text{for } x<0,\\
        \overline{F} \text{ is non-increasing and right-continuous on } [0,\infty)
      \end{array} 
    \right\}.
\]
For $\overline F\in\overline{\mathcal F}$ and $x\ge0$, we define the lower operator $T_{\mathrm{L}}$ by
\begin{equation}
\label{eq:TLdef}
    T_{\mathrm{L}}(\overline{F})(x)
    = \lambda \int_0^\infty \overline{B}_{\CP}(t) \overline{F}(x - t) \diff t
    + \lambda \int_0^\infty \overline{B}_{\CS}(t) \overline{F}(x - t) \diff t.
\end{equation}
The fixed-point equation $\overline{F} = T_{\mathrm{L}}(\overline F)$ is exactly the level-crossing equation for the stationary workload tail of a two-class $\mathrm{M/G/1}$ queue with independent Poisson arrivals of $\CP$ and $\CS$ customers. We denote its unique fixed-point solution by $\overline{F}_{V_{\mathrm{L}}}$. Similarly, we define the upper operator $T_{\mathrm{U}}$ as
\begin{equation}
\label{eq:TUdef}
    T_{\mathrm{U}}(\overline{F})(x)
    = \lambda \int_0^\infty \overline{B}_{\CP + \CS}(t) \overline{F}(x - t) \diff t.
\end{equation}
The fixed-point equation $\overline{F} = T_{\mathrm{U}}(\overline F)$ is the level-crossing equation for a batch-arrival $\mathrm{M^{[2]}/G/1}$ queue in which each primary--secondary pair arrives as a single job with service requirement $X_{\CS} + X_{\CP}$. We denote its stationary workload tail by $\overline{F}_{V_{\mathrm{U}}}$. With these operators, \eqref{eq:lower_bound} and \eqref{eq:upper_bound} can be compactly written as the operator inequality
\begin{equation}
\label{eq:operator_inequality}
    T_{\mathrm{L}}(\overline{F}_{V_{\mathrm{M}}})(x)
    \le \overline{F}_{V_{\mathrm{M}}}(x)
    \le T_{\mathrm{U}}(\overline{F}_{V_{\mathrm{M}}})(x),
    \qquad x \ge 0.
\end{equation}
\begin{lemma}
\label{lem:operator_properties}
Under the stability condition \eqref{eq:model:stability}, the operators $T_{\mathrm{L}}$ and $T_{\mathrm{U}}$ have the following properties:
\begin{enumerate}
    \item They are monotone: if $\overline F_1 \le \overline F_2$ pointwise, then
    \[
        T_{\mathrm{L}}(\overline F_1)
        \le T_{\mathrm{L}}(\overline F_2),
        \qquad
        T_{\mathrm{U}}(\overline F_1)
        \le T_{\mathrm{U}}(\overline F_2).
    \]

    \item 
    They are continuous under pointwise convergence: if $\overline F_n(x) \to \overline F(x)$ for every $x$, then, for every $x \ge 0$,
    \[
        T_{\mathrm{L}}(\overline F_n)(x) \to T_{\mathrm{L}}(\overline F)(x),
        \qquad
        T_{\mathrm{U}}(\overline F_n)(x) \to T_{\mathrm{U}}(\overline F)(x).
    \]

    \item The fixed-point equations $\overline F = T_{\mathrm{L}}(\overline F)$ and $\overline F = T_{\mathrm{U}}(\overline F)$ have unique solutions $\overline{F}_{V_{\mathrm{L}}}$ and $\overline{F}_{V_{\mathrm{U}}}$ in $\overline{\mathcal F}$, respectively.

\end{enumerate}
\end{lemma}

\begin{proof}
See Appendix~\ref{app:operator_properties_proof}.
\end{proof}
We now use the operator inequality~\eqref{eq:operator_inequality} to pass from inequalities to stochastic ordering.

\paragraph{Lower Bound Convergence}
We initialize the sequence with $\overline F_0 = \overline{F}_{V_{\mathrm{M}}}$ and iterate $\overline F_{n+1} = T_{\mathrm{L}}(\overline F_n)$ for $n \ge 0$.
By \eqref{eq:operator_inequality}, $\overline F_1 \le \overline F_0$. Applying the monotonicity property then yields
\[
    \overline F_0 \ge \overline F_1 \ge \overline F_2 \ge \cdots.
\]
The sequence is bounded below by zero and consequently converges pointwise to a limit function $\overline F_\infty$. By the continuity of $T_{\mathrm{L}}$, it follows that $\overline F_\infty = T_{\mathrm{L}}(\overline F_\infty)$.
The uniqueness of the lower fixed point implies $\overline F_\infty = \overline{F}_{V_{\mathrm{L}}}$, establishing that
\begin{equation}
\label{eq:lower_stochastic_order}
    V_{\mathrm{L}} \le_{\mathrm{st}} V_{\mathrm{M}}.
\end{equation}

\paragraph{Upper Bound Convergence}
Similarly, we set $\overline G_0 = \overline{F}_{V_{\mathrm{M}}}$ and iterate $\overline G_{n+1} = T_{\mathrm{U}}(\overline G_n)$ for $n \ge 0$.
Here \eqref{eq:operator_inequality} gives $\overline G_0 \le \overline G_1$, and monotonicity implies
\[
    \overline G_0 \le \overline G_1 \le \overline G_2 \le \cdots.
\]
The sequence is bounded above by one and converges pointwise to a fixed point of $T_{\mathrm{U}}$. Uniqueness implies this limit is $\overline{F}_{V_{\mathrm{U}}}$, establishing
\begin{equation}
\label{eq:upper_stochastic_order}
 V_{\mathrm{M}} \le_{\mathrm{st}} V_{\mathrm{U}}.
\end{equation}

Combining \eqref{eq:lower_stochastic_order} and \eqref{eq:upper_stochastic_order} establishes the workload sandwich
\begin{equation}
\label{eq:workload_bounds}
    V_{\mathrm{L}} \le_{\mathrm{st}} V_{\mathrm{M}} \le_{\mathrm{st}} V_{\mathrm{U}}.
\end{equation}
Thus, the true steady-state workload of the $\mathrm{M^T/G/1}$ queue lies between the steady-state workloads of two classical queues.

\subsection{Extreme Triggering Delays}
\label{sec:asymptotics}

The stochastic sandwich established previously, $V_{\mathrm{L}} \le_{\mathrm{st}} V_{\mathrm{M}} \le_{\mathrm{st}} V_{\mathrm{U}}$, provides bounds for the stationary workload distribution of the $\mathrm{M^T/G/1}$ queue under general triggering-delay and service-time distributions. In this subsection, we show that both bounds are asymptotically tight in the limit in the two extreme regimes of triggering delays.

Let $\{Y^{(k)}\}_{k \ge 1}$ be a sequence of non-negative random variables representing the triggering delays, indexed by $k = 1, 2, \dots$. For each fixed $k$, let $V_{\mathrm{M}}^{(k)}$ denote the stationary workload of the corresponding $\mathrm{M^T/G/1}$ queue operating under triggering delay $Y^{(k)}$.

\begin{theorem}
\label{thm:asymp_large_delay}
If the triggering delays diverge in probability to infinity, i.e.,
\[
    Y^{(k)} \xrightarrow{\mathrm{P}} \infty
    \qquad\text{as } k \to \infty,
\]
then the stationary workload converges in distribution to the lower-bound workload:
\[
    V_{\mathrm{M}}^{(k)} \xrightarrow{\mathrm{d}} V_{\mathrm{L}}, \qquad \text{as } k \to \infty.
\]
\end{theorem}

\begin{proof}
See Appendix~\ref{app:large_delay_proof}.
\end{proof}

\begin{theorem}
\label{thm:asymp_zero_delay}
If the triggering delays converge in probability to zero, i.e.,
\[
    Y^{(k)} \xrightarrow{\mathrm{P}} 0 \qquad\text{as } k \to \infty,
\]
then the stationary workload converges in distribution to the upper-bound workload:
\[
    V_{\mathrm{M}}^{(k)} \xrightarrow{\mathrm{d}} V_{\mathrm{U}}, \qquad \text{as } k \to \infty.
\]
\end{theorem}

\begin{proof}
See Appendix~\ref{app:small_delay_proof}.
\end{proof}

Theorems~\ref{thm:asymp_large_delay} and~\ref{thm:asymp_zero_delay} demonstrate that $V_{\mathrm{L}}$ and $V_{\mathrm{U}}$ are not merely bounding distributions; rather, they arise as the limiting stationary workload distributions in the two extreme regimes of the triggering delay. When triggering delays diverge to infinity in probability, the stationary workload distribution converges to that of the independent two-class $\mathrm{M/G/1}$ queue. Conversely, when triggering delays converge to zero in probability, primary and secondary customers tend to arrive as a batch, and the stationary workload distribution converges to that of the $\mathrm{M^{[2]}/G/1}$ queue.  It is important to note that these limit statements do not imply that $V_{\mathrm{M}}^{(k)}$ is generally stochastically monotone in $Y^{(k)}$.

\subsection{System-Time Bounds}
\label{subsec:system_time_bounds}

We finally translate the workload distribution sandwich into bounds on class-specific system times. For any tagged customer, the system time equals the workload observed immediately before arrival plus its own service requirement.

\paragraph{Primary customers}
For primary customers, PASTA gives $W_{\CP} \overset{\mathrm{d}}{=} V_{\mathrm{M}}$, and the tagged service requirement is independent of the pre-arrival workload. Therefore,
\[
    S_{\CP} \overset{\mathrm{d}}{=} V_{\mathrm{M}} + X_{\CP}^\star.
\]
Combining this representation with \eqref{eq:workload_bounds}, and noting that $X_{\CP}^\star$ is independent of the comparison workloads, yields
\begin{equation}
\label{eq:primary_time_bound}
    V_{\mathrm{L}} + X_{\CP}^\star
    \le_{\mathrm{st}} S_{\CP}
    \le_{\mathrm{st}} V_{\mathrm{U}} + X_{\CP}^\star.
\end{equation}
The lower bound corresponds to the system time of a primary customer in the two-class $\mathrm{M/G/1}$ queue, whereas the upper bound corresponds to the system time of a primary customer in the $\mathrm{M^{[2]}/G/1}$ queue.

\paragraph{Secondary customers}
For secondary customers, the arrival epoch is not an independent observation epoch, so we again use the coupling result from Section~\ref{sec:ph}. Under the coupling of Section~\ref{sec:ph}, we have
\[
  V_{\mathrm{M}}^{\mathrm{RP}}(Y^-)
   \le V_{\mathrm{M}}^{\mathrm{TP}}(Y^-)
   \le V_{\mathrm{M}}^{\mathrm{RP}}(Y^-) + X_{\CP}^{\star}
  \qquad \text{a.s.}
\]
Adding $X_{\CS}^{\star}$ to these pathwise inequalities gives
\[
  V_{\mathrm{M}}^{\mathrm{RP}}(Y^-) + X_{\CS}^{\star}
   \le V_{\mathrm{M}}^{\mathrm{TP}}(Y^-) + X_{\CS}^{\star}
   \le V_{\mathrm{M}}^{\mathrm{RP}}(Y^-) + X_{\CP}^{\star} + X_{\CS}^{\star}
  \qquad \text{a.s.}
\]
Since $V_{\mathrm{M}}^{\mathrm{TP}}(Y^-) + X_{\CS}^{\star} \stackrel{\mathrm{d}}{=}  S_{\CS}$, and $V_{\mathrm{M}}^{\mathrm{RP}}(Y^-) \stackrel{\mathrm{d}}{=} V_{\mathrm{M}}$ is independent of the tagged service pair $\bigl(X_{\CP}^{\star}, X_{\CS}^{\star}\bigr)$, it follows that
\[
  V_{\mathrm{M}} + X_{\CS}^{\star}
   \le_{\mathrm{st}}  S_{\CS}
   \le_{\mathrm{st}}  V_{\mathrm{M}} + X_{\CP}^{\star} + X_{\CS}^{\star}.
\]
Finally, applying the workload sandwich~\eqref{eq:workload_bounds}, with the comparison workloads taken independently of the tagged service pair, yields
\begin{equation}
  \label{eq:secondary_time_bound}
  V_{\mathrm{L}} + X_{\CS}^{\star}
   \le_{\mathrm{st}}  S_{\CS}
   \le_{\mathrm{st}}  V_{\mathrm{U}} + X_{\CP}^{\star} + X_{\CS}^{\star}.
\end{equation}
The lower bound in~\eqref{eq:secondary_time_bound} is the system-time distribution of a secondary customer in the independent two-class $\mathrm{M/G/1}$ queue. The upper bound is the system-time distribution of a secondary customer in the $\mathrm{M^{[2]}/G/1}$ queue when the primary customer is served before the corresponding secondary customer within each batch.

We summarize these results in the following theorem.
\begin{theorem}\label{thm:system_time_bounds}
Consider the $\mathrm{M^T/G/1}$ queue with generally distributed triggering delays. The stationary system times of primary and secondary customers, $S_{\CP}$ and $S_{\CS}$, satisfy
\[
    S_{\CP,\mathrm{L}} \le_{\mathrm{st}} S_{\CP} \le_{\mathrm{st}} S_{\CP,\mathrm{U}},
    \qquad
    S_{\CS,\mathrm{L}} \le_{\mathrm{st}} S_{\CS} \le_{\mathrm{st}} S_{\CS,\mathrm{U}},
\]
where $S_{\CP,\mathrm{L}}$ and $S_{\CS,\mathrm{L}}$ denote the respective system times in the two-class $\mathrm{M/G/1}$ queue, and $S_{\CP,\mathrm{U}}$ and $S_{\CS,\mathrm{U}}$ denote the respective system times in the batch-arrival $\mathrm{M^{[2]}/G/1}$ queue.
Moreover, these bounds are asymptotically tight in the limiting regimes established in Theorems~\ref{thm:asymp_large_delay} and~\ref{thm:asymp_zero_delay}. In particular,
\[
    S_{\CP}^{(k)} \xrightarrow{\mathrm{d}} S_{\CP,\mathrm{L}},
    \qquad
    S_{\CS}^{(k)} \xrightarrow{\mathrm{d}} S_{\CS,\mathrm{L}},
    \qquad
    \text{as } Y^{(k)} \xrightarrow{\mathrm{P}} \infty 
    \quad
    (k\to\infty),
\]
whereas
\[
    S_{\CP}^{(k)} \xrightarrow{\mathrm{d}} S_{\CP,\mathrm{U}},
    \qquad
    S_{\CS}^{(k)} \xrightarrow{\mathrm{d}} S_{\CS,\mathrm{U}},
    \qquad
    \text{as } Y^{(k)} \xrightarrow{\mathrm{P}} 0
    \quad
    (k\to\infty).
\]
\end{theorem}

\subsection{Which Bound Applies? A Heuristic Engineering Approximation}
\label{subsec:which_bound}

We begin by addressing two practical questions: How large is the gap between the bounds, and which bound is likely to provide a more accurate approximation at a given operating point?

To address these questions, we use the coupling construction established in Section~\ref{sec:ph}. Let $T_0$ denote the \textit{first} time at which the queue becomes empty after the arrival of the tagged primary customer in TP:
\begin{equation}
T_0 \coloneqq \inf \left\{ s > 0 : V_{\mathrm{M}} + X^\star_{\CP} + \bar{A}(s) - s \le 0 \right\}.
\end{equation}
By \eqref{eq:delta_identity}, $Y>T_0$ implies $\Delta=0$. Thus, if the triggering delay exceeds $T_0$, the tagged secondary customer arrives only after the queue has emptied at least once following the tagged primary arrival. Moreover, $Y<T_0-X^\star_{\CP}$ implies $\Delta=X^\star_{\CP}$, although the converse need not hold.

The distribution of $T_0$ is challenging to characterize analytically. Furthermore, in practical systems, the distribution of the triggering delay $Y$ may itself be difficult to determine, since the triggering mechanism can depend on complex internal workflows.

To obtain a practical approximation, we introduce the following heuristic. Motivated by the coupling construction, we assume, without a formal proof, that
$T_{\mathrm{L}}<T_0<T_{\mathrm{U}}$, where $T_{\mathrm{L}}$ denotes the first time the lower-bound queue becomes empty after the tagged primary arrival, and $T_{\mathrm{U}}$ denotes the first time the upper-bound queue downcrosses the workload level $X^\star_{\CP}$. Under this assumption, $Y>T_{\mathrm{U}}$ implies $\Delta=0$, whereas $Y<T_{\mathrm{L}}-X^\star_{\CP}$ implies $\Delta=X^\star_{\CP}$. We now introduce a \textit{first-moment rule of thumb} as a simple engineering guideline.

In an $\mathrm{M/G/1}$ queue, the expected time required for the queue to become empty, given an initial workload $W_0$, is $\mathbb{E}[W_0]/(1-\rho)$; this is the mean delay-cycle duration~\cite{Kleinrock1976Volume2}. Applying this result to the lower- and upper-bound queues immediately after a tagged primary arrival gives
\begin{equation}\label{eq:10}
    \mathbb{E}[T_{\mathrm{L}}] = \frac{\lambda\bigl(\mathbb{E}[X_{\CP}^2] +\mathbb{E}[X_{\CS}^2]\bigr)}{2(1-\rho)^2} + \frac{\mathbb{E}[X_{\CP}]}{1-\rho},\qquad
    \mathbb{E}[T_{\mathrm{U}}] = \frac{\lambda\mathbb{E}[(X_{\CP}+X_{\CS})^2]}{2(1-\rho)^2} + \frac{\mathbb{E}[X_{\CP}]}{1-\rho},
\end{equation}
where $\mathbb{E}[T_{\mathrm{L}}]\leq \mathbb{E}[T_{\mathrm{U}}]$ whenever $\mathbb{E}[X_{\CP}X_{\CS}]\geq 0$. This motivates comparing the mean triggering delay $\mathbb{E}[Y]$ with $\mathbb{E}[T_{\mathrm{L}}]$ and $\mathbb{E}[T_{\mathrm{U}}]$:
\begin{align*}
    \mathbb{E}[Y] \ll \mathbb{E}[T_{\mathrm{L}}]
    &\implies \text{the upper-bound queue is expected to provide a good approximation},\\
    \mathbb{E}[Y] \gg \mathbb{E}[T_{\mathrm{U}}]
    &\implies \text{the lower-bound queue is expected to provide a good approximation}.
\end{align*}
When $\mathbb{E}[Y]$ is comparable to both $\mathbb{E}[T_{\mathrm{L}}]$ and $\mathbb{E}[T_{\mathrm{U}}]$, neither bound is expected to be clearly preferable based on this rule alone.

\section{Numerical Results}
\label{sec:numerical}

\begin{figure}[t]
    \centering
    \subfloat[Deterministic triggering delays ($C_Y^2 = 0$).\label{fig:bounds:dependent:det}]{\input{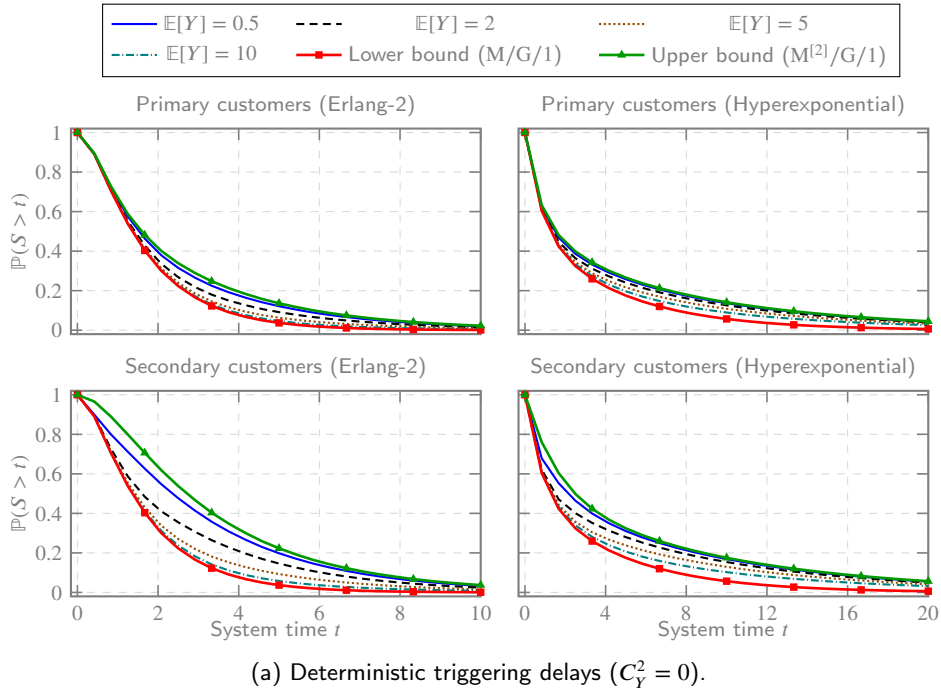}}\\
    \subfloat[Log-normal triggering delays ($C_Y^2 = 10$).\label{fig:bounds:dependent:lnorm}]{\input{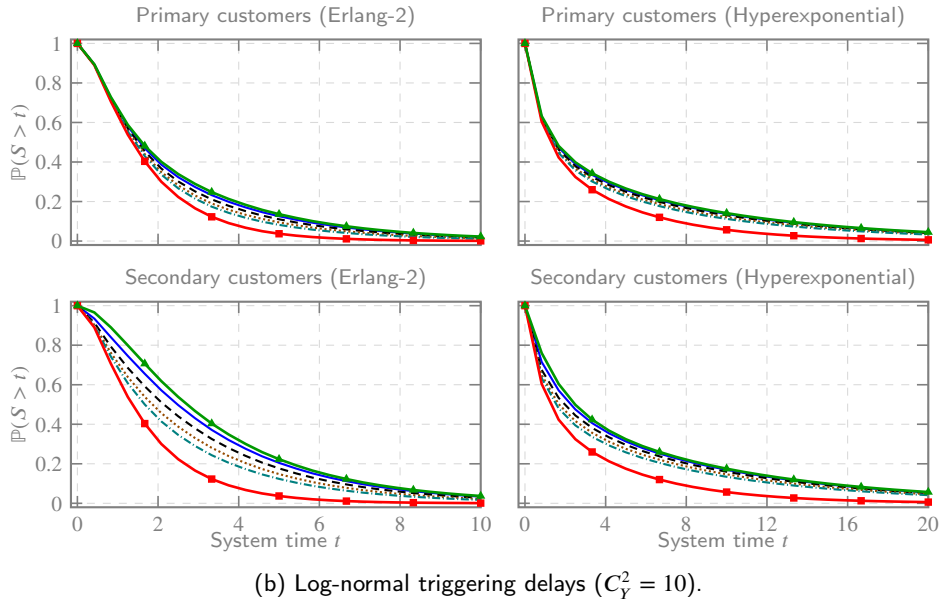}}
    \caption{Tail distribution of system time for primary ($\CP$) and secondary ($\CS$) customers under dependent service times ($X_{\CP} = X_{\CS}$, $\E[X]=1$, $\rho=0.5$). Curves are presented for Erlang-$2$ and hyperexponential service times across $\E[Y]\in\{0.5,\, 2,\, 5,\, 10\}$. Here $C_Y^2$ denotes the squared coefficient of variation of the triggering delay. The legend in (a) applies to (b).}
    \label{fig:bounds:dependent}
\end{figure}

In this section, we present numerical validations of the theoretical bounds and approximations established in Section~\ref{sec:bounds}. The evaluation has two main objectives: first, to demonstrate that the system-time distributions of both primary and secondary customers are stochastically sandwiched between those of the lower-bound (two-class $\mathrm{M/G/1}$) and upper-bound ($\mathrm{M^{[2]}/G/1}$) queues under general service-time and triggering-delay distributions; and second, to empirically validate the heuristic engineering approximation formulated in Section~\ref{subsec:which_bound}.

\subsection{Verification of System-Time Bounds}
\label{subsec:numerical:bounds}

We evaluate the tail distributions of the system times, $\Prob(S_{\CP} > t)$ and $\Prob(S_{\CS} > t)$, under two triggering-delay distributions with the same mean, $\E[Y]$, but contrasting levels of variability: \emph{deterministic} triggering delays, with a squared coefficient of variation of $C_Y^2 = 0$, and \emph{log-normal} triggering delays, with parameters selected to yield $C_Y^2 = 10$, representing high variability. The mean triggering delay is varied over $\E[Y] \in \{0.5,\, 2,\, 5,\, 10\}$.

To assess the effects of service-time distributions beyond the exponential setting, we examine two service-time distributions with unit means ($\E[X_{\CP}] = \E[X_{\CS}] = 1$): \emph{Erlang-$2$} service times ($C_X^2 = 0.5$, consisting of two stages, each with rate $2$), which represent low variability; and \emph{hyperexponential} service times ($C_X^2 = 2.5$), a two-phase mixture ($p_1/\mu_1 = p_2/\mu_2 = 0.5$, where $p_1 = \frac{1}{2}(1 + \sqrt{3/7})$ and $\mu_i = 2p_i$), which represents high variability. Within each primary--secondary pair, service requirements are dependent, with $X_{\CP} = X_{\CS}$ almost surely. The traffic load (utilization factor) is set to $\rho = 2\lambda = 0.5$. Each simulation run processes $3\times 10^7$ customer pairs. In each run, the first and last $10^6$ customer arrivals are discarded to eliminate initialization and termination effects. Analytical bounds are computed through numerical inversion of the Laplace--Stieltjes transform in Python.

Fig.~\ref{fig:bounds:dependent} is consistent with Theorem~\ref{thm:system_time_bounds}. The system-time distributions obtained through simulation consistently lie between the analytical lower bound (two-class $\mathrm{M/G/1}$) and upper bound ($\mathrm{M^{[2]}/G/1}$). Furthermore, the tail distribution curves for deterministic ($C_Y^2 = 0$) and log-normal ($C_Y^2 = 10$) triggering delays, together with Erlang and hyperexponential service-time distributions, illustrate the effects of the triggering-delay and service-time distributions and remain within the bounds. Moreover, the curves show that decreasing $\E[Y]$ shifts the tail distributions toward the upper bound, whereas increasing $\E[Y]$ shifts them toward the lower bound. This behavior is consistent with the asymptotic limits in Theorems~\ref{thm:asymp_large_delay} and~\ref{thm:asymp_zero_delay}.

The tails of these curves are difficult to read accurately from Fig.~\ref{fig:bounds:dependent}. However, they are important for mission-critical applications, such as ultra-reliable low-latency communication (URLLC) and industrial IoT alarm workflows, where performance guarantees rely on tail percentiles (e.g., 90th, 95th, and 99th). Accordingly, Table~\ref{tab:urllc:percentiles} presents system-time tail quantiles under a load of $\rho = 0.8$ for exponential ($C_X^2 = 1.0$) and hyperexponential ($C_X^2 = 2.5$) service times with deterministic triggering delays. For these percentiles, each simulated row is based on a single continuous run of $5\times 10^7$ customers, while the bound results are obtained analytically.

\begin{table*}[t]
\centering
\caption{Tail system-time quantiles (90th, 95th, and 99th percentiles; time units under $\E[X]=1$) for primary ($\CP$) and secondary ($\CS$) customers under heavy load ($\rho=0.8$), with dependent service requirements ($X_{\CP}=X_{\CS}$). ``Det'' denotes deterministic triggering delays. The simulated $\mathrm{M^T/G/1}$ cases use $\E[Y]\in\{1,\, 5,\, 20\}$. Analytical bounds are computed via numerical Laplace--Stieltjes transform inversion. Each simulated row is based on one continuous simulation run of $5\times 10^7$ customers.}
\label{tab:urllc:percentiles}
\footnotesize
\renewcommand{\arraystretch}{1.10}
\begin{tabular*}{\textwidth}{@{\extracolsep{\fill}}llccc@{}}
\toprule
\textbf{Service} & \textbf{Condition} & \textbf{90th ($\CP\,/\,\CS$)} & \textbf{95th ($\CP\,/\,\CS$)} & \textbf{99th ($\CP\,/\,\CS$)} \\
\midrule
\rowcolor{red!8}
\cellcolor{white}
\multirow{5}{*}{\shortstack[l]{\textbf{Exponential}\\$C_X^2 = 1.0$}}
& Lower bound (two-class $\mathrm{M/G/1}$)  & 11.51 / 11.51 & 14.98 / 14.98 & 23.03 / 23.03 \\
& $\mathrm{M^T/G/1}$ (Det, $\E[Y]=1.0$)     & 21.36 / 22.54 & 28.33 / 29.50 & 44.44 / 45.62 \\
& $\mathrm{M^T/G/1}$ (Det, $\E[Y]=5.0$)     & 19.90 / 21.08 & 26.81 / 27.99 & 42.90 / 44.08 \\
& $\mathrm{M^T/G/1}$ (Det, $\E[Y]=20.0$)    & 16.84 / 17.99 & 23.73 / 24.91 & 39.80 / 40.99 \\
\rowcolor{green!8}
\cellcolor{white}
& Upper bound ($\mathrm{M^{[2]}/G/1}$)      & 21.85 / 23.03 & 28.78 / 29.96 & 44.87 / 46.05 \\
\midrule
\rowcolor{red!8}
\cellcolor{white}
\multirow{5}{*}{\shortstack[l]{\textbf{Hyper-}\\\textbf{exponential}\\$C_X^2 = 2.5$}}
& Lower bound (two-class $\mathrm{M/G/1}$)  & 20.14 / 20.14 & 26.78 / 26.78 & 42.21 / 42.21 \\
& $\mathrm{M^T/G/1}$ (Det, $\E[Y]=1.0$)     & 38.36 / 39.64 & 51.63 / 52.91 & 82.41 / 83.68 \\
& $\mathrm{M^T/G/1}$ (Det, $\E[Y]=5.0$)     & 36.94 / 38.22 & 50.22 / 51.49 & 80.99 / 82.25 \\
& $\mathrm{M^T/G/1}$ (Det, $\E[Y]=20.0$)    & 33.14 / 34.41 & 46.42 / 47.69 & 77.32 / 78.57 \\
\rowcolor{green!8}
\cellcolor{white}
& Upper bound ($\mathrm{M^{[2]}/G/1}$)      & 39.00 / 40.28 & 52.29 / 53.57 & 83.15 / 84.43 \\
\bottomrule
\end{tabular*}
\end{table*}

As the table shows, all simulated percentiles remain strictly bounded between those of the lower and upper bounds. These results confirm that the bounds provide valuable tools for latency-critical performance evaluation in systems where triggered arrivals are prevalent.

\subsection{Validation of the Heuristic Engineering Approximation}
\label{subsec:numerical:timescale}

To assess our heuristic engineering approximation, we introduce the following metric:
\begin{equation}
    \eta(\E[Y]) \coloneqq
    \frac{\E[S_{\CS}] - \E[S_{\CS,\mathrm{L}}]}
         {\E[S_{\CS,\mathrm{U}}] - \E[S_{\CS,\mathrm{L}}]}
    \in [0, 1].
    \label{eq:numerical:eta}
\end{equation}
As $\E[S_{\CS}]$ approaches $\E[S_{\CS,\mathrm{U}}]$, $\eta \to 1$, indicating that the upper bound is a good approximation. Conversely, as $\E[S_{\CS}]$ approaches $\E[S_{\CS,\mathrm{L}}]$, $\eta \to 0$, indicating that the lower bound is a good approximation. For $X_{\CP} = X_{\CS} = X$ (a.s.) with $\E[X] = 1$ and $\rho = 0.5$, we obtain $\E[T_{\mathrm{L}}] = 3.5$ and $\E[T_{\mathrm{U}}] = 5.0$ for Erlang-$2$ service times, and $\E[T_{\mathrm{L}}] = 5.5$ and $\E[T_{\mathrm{U}}] = 9.0$ for hyperexponential service times.

Fig.~\ref{fig:numerical:timescale} plots $\eta(\E[Y])$ against the mean triggering delay $\E[Y]$ for deterministic, exponential, and log-normal triggering delays, with squared coefficients of variation $C_Y^2=0$, $1$, and $10$, respectively. For both service-time distributions, $\eta$ decreases as $\E[Y]$ increases, reflecting the transition from the batch-arrival upper-bound regime at short triggering delays toward the independent-arrival lower-bound regime at long delays. The transition depends on both the service-time distribution and the triggering-delay distribution.

To indicate a threshold at which each bound is expected to provide a useful approximation, we use the practical criteria
\[
    \E[Y]\le \frac{\E[T_{\mathrm{L}}]}{10}
    \qquad\text{and}\qquad
    \E[Y]\ge 10\,\E[T_{\mathrm{U}}],
\]
for which the upper- and lower-bound regimes provide good approximations, respectively. These regimes are heuristic guidelines rather than formal error guarantees. In Fig.~\ref{fig:numerical:timescale}, the shaded green bands denote the heuristic upper-bound regime, while the shaded blue bands denote the lower-bound regime. The corresponding theoretical values of $\mathbb{E}[T_{\mathrm{L}}]$ and $\mathbb{E}[T_{\mathrm{U}}]$ are shown in each panel.

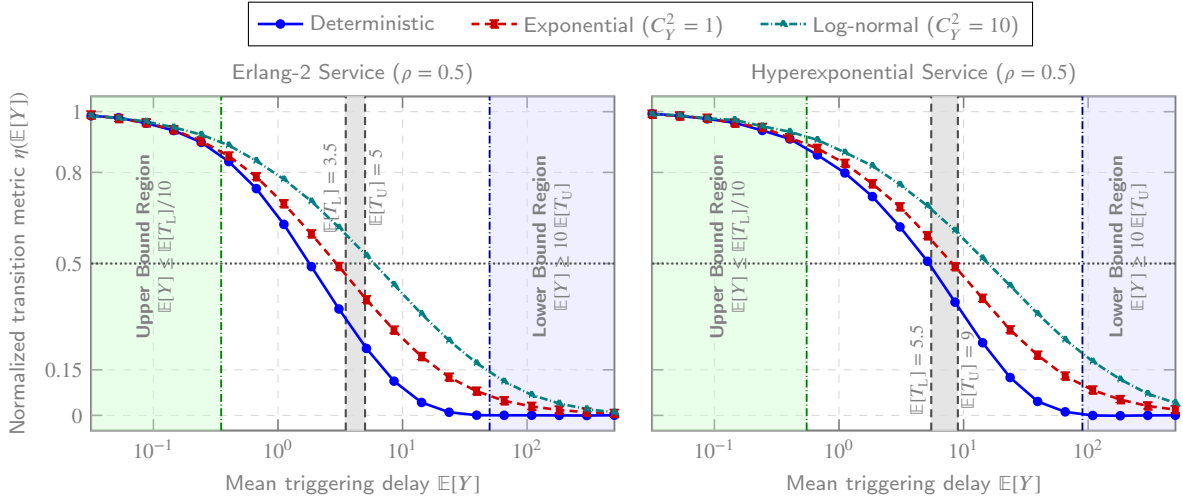
\begin{figure}[t]
\centering
\begin{tikzpicture}
  \begin{groupplot}[
    group style={
      group size=2 by 1,
      horizontal sep=0.5cm,
      y descriptions at=edge left,
      x descriptions at=edge bottom,
    },
    width=8.5cm, height=6.0cm,
    grid=major, grid style={dashed, gray!30},
    xmode=log, log basis x=10,
    ymin=-0.05, ymax=1.05, xmin=0.0316, xmax=501.1872,
    xlabel={Mean triggering delay $\E[Y]$},
    ylabel={Normalized transition metric $\eta(\E[Y])$},
    tick label style={font=\footnotesize},
    label style={font=\footnotesize},
    title style={font=\footnotesize, yshift=-0.5em},
    axis line style={line width=0.8pt},
    tick style={line width=0.6pt},
    ytick={0, 0.15, 0.5, 0.8, 1.0},
    xtick={0.1, 1, 10, 100},
    xticklabels={$10^{-1}$, $10^{0}$, $10^{1}$, $10^{2}$}
  ]

  \nextgroupplot[title={Erlang-2 Service ($\rho=0.5$)},
    legend columns=3,
    legend style={font=\footnotesize, at={(1.05, 1.15)}, anchor=south,
      /tikz/every even column/.append style={column sep=0.4cm}}]
    \fill[green!20, opacity=0.4] (axis cs:0.0316, -0.05) rectangle (axis cs:0.3500, 1.05);
    \fill[gray!25,  opacity=0.8] (axis cs:3.5, -0.05) rectangle (axis cs:5.0, 1.05);
    \fill[blue!15,  opacity=0.4] (axis cs:50.0000, -0.05) rectangle (axis cs:501.1872, 1.05);

    \draw[densely dotted, thick, black!70]
      (axis cs:0.0316, 0.5) -- (axis cs:501.1872, 0.5);
    \draw[densely dashed, line width=0.8pt, black!70]
      (axis cs:3.5, -0.05) -- (axis cs:3.5, 1.05);
    \draw[densely dashed, line width=0.8pt, black!70]
      (axis cs:5.0, -0.05) -- (axis cs:5.0, 1.05);
    \draw[densely dashdotted, line width=0.7pt, green!50!black]
      (axis cs:0.3500, -0.05) -- (axis cs:0.3500, 1.05);
    \draw[densely dashdotted, line width=0.7pt, blue!50!black]
      (axis cs:50.0000, -0.05) -- (axis cs:50.0000, 1.05);

    \node[font=\scriptsize, rotate=90, anchor=center, align=center]
      at (rel axis cs:0.124, 0.55)
      {\textbf{Upper Bound Region}\\$\E[Y] \le \E[T_{\mathrm{L}}]/10$};
    \node[font=\scriptsize, rotate=90, anchor=center, align=center]
      at (rel axis cs:0.881, 0.55)
      {\textbf{Lower Bound Region}\\$\E[Y] \ge 10\,\E[T_{\mathrm{U}}]$};

    \node[font=\scriptsize, rotate=90, anchor=south, fill=white, fill opacity=0.7,
          text opacity=1, inner sep=0.5pt]
      at (axis cs:3.3, 0.75) {$\E[T_{\mathrm{L}}]=3.5$};
    \node[font=\scriptsize, rotate=90, anchor=north, fill=white, fill opacity=0.7,
          text opacity=1, inner sep=0.5pt]
      at (axis cs:5.2, 0.75) {$\E[T_{\mathrm{U}}]=5$};

    \addplot[blue, line width=1pt, mark=*, mark size=1.3pt] coordinates { (0.0316,0.9864) (0.0526,0.9798) (0.0875,0.9623) (0.1456,0.9379) (0.2422,0.8988) (0.4030,0.8355) (0.6704,0.7466) (1.1152,0.6289) (1.8553,0.4898) (3.0866,0.3503) (5.1348,0.2202) (8.5424,0.1126) (14.2112,0.0425) (23.6420,0.0103) (39.3312,0.0006) (65.4319,0.0002) (108.8534,0.0003) (181.0901,-0.0001) (301.2641,-0.0006) (501.1872,0.0005) };
    \addplot[red!80!black, densely dashed, line width=1pt,
             mark=square*, mark size=1.3pt] coordinates { (0.0316,0.9888) (0.0526,0.9773) (0.0875,0.9629) (0.1456,0.9402) (0.2422,0.9034) (0.4030,0.8538) (0.6704,0.7857) (1.1152,0.6966) (1.8553,0.5973) (3.0866,0.4900) (5.1348,0.3809) (8.5424,0.2805) (14.2112,0.1937) (23.6420,0.1256) (39.3312,0.0796) (65.4319,0.0479) (108.8534,0.0295) (181.0901,0.0174) (301.2641,0.0107) (501.1872,0.0062) };
    \addplot[teal, densely dashdotted, line width=1pt,
             mark=triangle*, mark size=1.3pt] coordinates { (0.0316,0.9868) (0.0526,0.9790) (0.0875,0.9675) (0.1456,0.9483) (0.2422,0.9242) (0.4030,0.8892) (0.6704,0.8381) (1.1152,0.7773) (1.8553,0.7048) (3.0866,0.6186) (5.1348,0.5263) (8.5424,0.4303) (14.2112,0.3339) (23.6420,0.2465) (39.3312,0.1720) (65.4319,0.1117) (108.8534,0.0674) (181.0901,0.0391) (301.2641,0.0207) (501.1872,0.0099) };
    
    \addlegendentry{Deterministic}
    \addlegendentry{Exponential ($C_Y^2=1$)}
    \addlegendentry{Log-normal ($C_Y^2=10$)}

  \nextgroupplot[title={Hyperexponential Service ($\rho=0.5$)}]
    \fill[green!20, opacity=0.4] (axis cs:0.0316, -0.05) rectangle (axis cs:0.5500, 1.05);
    \fill[gray!25,  opacity=0.8] (axis cs:5.5, -0.05) rectangle (axis cs:9.0, 1.05);
    \fill[blue!15,  opacity=0.4] (axis cs:90.0000, -0.05) rectangle (axis cs:501.1872, 1.05);

    \draw[densely dotted, thick, black!70]
      (axis cs:0.0316, 0.5) -- (axis cs:501.1872, 0.5);
    \draw[densely dashed, line width=0.8pt, black!70]
      (axis cs:5.5, -0.05) -- (axis cs:5.5, 1.05);
    \draw[densely dashed, line width=0.8pt, black!70]
      (axis cs:9.0, -0.05) -- (axis cs:9.0, 1.05);
    \draw[densely dashdotted, line width=0.7pt, green!50!black]
      (axis cs:0.5500, -0.05) -- (axis cs:0.5500, 1.05);
    \draw[densely dashdotted, line width=0.7pt, blue!50!black]
      (axis cs:90.0000, -0.05) -- (axis cs:90.0000, 1.05);

    \node[font=\scriptsize, rotate=90, anchor=center, align=center]
      at (rel axis cs:0.148, 0.55)
      {\textbf{Upper Bound Region}\\$\E[Y] \le \E[T_{\mathrm{L}}]/10$};
    \node[font=\scriptsize, rotate=90, anchor=center, align=center]
      at (rel axis cs:0.911, 0.55)
      {\textbf{Lower Bound Region}\\$\E[Y] \ge 10\,\E[T_{\mathrm{U}}]$};

    \node[font=\scriptsize, rotate=90, anchor=south, fill=white, fill opacity=0.7,
          text opacity=1, inner sep=0.5pt]
      at (axis cs:5.3, 0.15) {$\E[T_{\mathrm{L}}]=5.5$};
    \node[font=\scriptsize, rotate=90, anchor=north, fill=white, fill opacity=0.7,
          text opacity=1, inner sep=0.5pt]
      at (axis cs:9.2, 0.15) {$\E[T_{\mathrm{U}}]=9$};

    \addplot[blue, line width=1pt, mark=*, mark size=1.3pt] coordinates { (0.0316,0.9931) (0.0526,0.9860) (0.0875,0.9765) (0.1456,0.9632) (0.2422,0.9381) (0.4030,0.9101) (0.6704,0.8575) (1.1152,0.7982) (1.8553,0.7204) (3.0866,0.6205) (5.1348,0.5072) (8.5424,0.3728) (14.2112,0.2392) (23.6420,0.1246) (39.3312,0.0457) (65.4319,0.0116) (108.8534,-0.0004) (181.0901,-0.0012) (301.2641,-0.0001) (501.1872,0.0003) };
    \addplot[red!80!black, densely dashed, line width=1pt,
             mark=square*, mark size=1.3pt] coordinates { (0.0316,0.9916) (0.0526,0.9857) (0.0875,0.9782) (0.1456,0.9622) (0.2422,0.9478) (0.4030,0.9139) (0.6704,0.8790) (1.1152,0.8295) (1.8553,0.7622) (3.0866,0.6864) (5.1348,0.5912) (8.5424,0.4895) (14.2112,0.3853) (23.6420,0.2816) (39.3312,0.1981) (65.4319,0.1293) (108.8534,0.0833) (181.0901,0.0518) (301.2641,0.0310) (501.1872,0.0188) };
    \addplot[teal, densely dashdotted, line width=1pt,
             mark=triangle*, mark size=1.3pt] coordinates { (0.0316,0.9916) (0.0526,0.9880) (0.0875,0.9779) (0.1456,0.9727) (0.2422,0.9519) (0.4030,0.9338) (0.6704,0.9069) (1.1152,0.8658) (1.8553,0.8214) (3.0866,0.7598) (5.1348,0.6897) (8.5424,0.6088) (14.2112,0.5173) (23.6420,0.4273) (39.3312,0.3343) (65.4319,0.2493) (108.8534,0.1776) (181.0901,0.1187) (301.2641,0.0711) (501.1872,0.0411) };

  \end{groupplot}
\end{tikzpicture}
\caption{Validation of the heuristic approximation using the metric $\eta(\mathbb{E}[Y])$ in \eqref{eq:numerical:eta} as a function of the mean triggering delay $\mathbb{E}[Y]$, for Erlang-2 ($C_X^2=0.5$) and hyperexponential ($C_X^2=2.5$) service-time distributions. Here, $\eta \to 1$ indicates convergence to the batch-arrival $\mathrm{M^{[2]}/G/1}$ upper bound, while $\eta \to 0$ indicates convergence to the independent two-class $\mathrm{M/G/1}$ lower bound. Curves correspond to deterministic ($C_Y^2=0$), exponential ($C_Y^2=1$), and log-normal ($C_Y^2=10$) triggering delays. We use $\mathbb{E}[Y]\le \mathbb{E}[T\_{\mathrm{L}}]/10$ and $\mathbb{E}[Y]\ge 10,\mathbb{E}[T\_{\mathrm{U}}]$ as practical criteria for the heuristic approximation to approach the upper and lower bounds, respectively. The shaded green bands indicate the heuristic upper-bound regime, whereas the shaded blue bands indicate the lower-bound regime. The corresponding theoretical values of $\mathbb{E}[T_{\mathrm{L}}]$ and $\mathbb{E}[T_{\mathrm{U}}]$ are indicated in each panel.}
\label{fig:numerical:timescale}
\end{figure}

\section{Conclusion}
\label{sec:conclusion}

In this paper, we investigated the triggered-arrival queueing system, where each exogenous primary arrival triggers a secondary arrival after a general, independent triggering delay. In this queueing system, an exact Markovian characterization under general triggering delays and general service requirements generally requires tracking an unbounded number of continuous residual lifetimes of the triggering delay, which leads to an infinite-dimensional state space.

To overcome this bottleneck, we developed a sample-path coupling framework based on an auxiliary $\mathrm{M/G/\infty}$ delay line populated by an independent stream of reference Poisson customers ($\CA$). Exploiting the Skorokhod reflection mapping over the triggering interval, we established an exact pathwise relationship showing that an arriving secondary customer encounters a workload almost surely bounded between the coupled reference workload and that workload augmented by the service requirement of its corresponding primary customer. Combining this sample-path ordering with level-crossing arguments, we proved that the stationary workload $V_{\mathrm M}$ of the $\mathrm{M^T/G/1}$ queue is stochastically sandwiched between the stationary workloads of two classical, tractable queueing models: $V_{\mathrm L} \le_{\mathrm{st}} V_{\mathrm M} \le_{\mathrm{st}} V_{\mathrm U}$, where $V_{\mathrm L}$ corresponds to a two-class $\mathrm{M/G/1}$ queue with independent Poisson arrivals, and $V_{\mathrm U}$ corresponds to a batch-arrival $\mathrm{M^{[2]}/G/1}$ queue in which each primary--secondary pair enters simultaneously. This stochastic sandwich holds under general triggering-delay distributions and general within-pair service distributions satisfying stability and finite second moments. We further showed that the bounds are asymptotically tight: $V_{\mathrm M} \xrightarrow{\mathrm d} V_{\mathrm L}$ as triggering delays diverge in probability to infinity, and $V_{\mathrm M} \xrightarrow{\mathrm d} V_{\mathrm U}$ as they vanish in probability to zero. We then translated these results into rigorous stochastic bounds on class-specific system-time distributions.

From an engineering perspective, our results demonstrate that when a conservative, worst-case performance guarantee is sufficient, system designers do not need to estimate or track the triggering-delay distribution: the batch-arrival $\mathrm{M^{[2]}/G/1}$ envelope provides a rigorous upper bound on delays and workload exceedance probabilities in the infinite-buffer model. However, when the triggering delay is comparable to the queue clearing time, neither bound is guaranteed to be tight; in this regime, incorporating the triggering-delay timescale $\E[Y]$ relative to the queue clearing time may help obtain a tighter approximation.



\appendix
\makeatletter
\renewcommand*{\@seccntformat}[1]{%
  \expandafter\ifx\csname c@#1\endcsname\c@section
    Appendix~%
  \fi
  \csname the#1\endcsname.\quad
}
\makeatother
\section{Proof of Lemma~\ref{lem:ph:poisson_dep}}
\label{app:poisson_dep}
We aim to demonstrate that the departure stream of auxiliary customers $\CA$ from $\mathrm{Q_A}$ is a causally independent Poisson process with respect to the main queue $\mathrm{Q_M}$. By the displacement theorem for Poisson point processes~\cite{baccelli2013elements,daley2003introduction}, applying i.i.d. non-negative translation delays $\{Y_i\}$ to the points of an exogenous Poisson process with rate $\lambda_{\CA} > 0$ yields another homogeneous Poisson process with the same intensity, $\lambda_{\CA}$. Because $\mathrm{Q_A}$ is an infinite-server facility ($\mathrm{M/G/\infty}$), auxiliary customers experience no queueing delay and do not interact with other customers. Moreover, their arrival epochs and triggering delays are mutually independent of all service times $\{(X_{\CP, i}, X_{\CS, i})\}$ and of the primary arrival epochs $\{T_i\}$. Therefore, the departure stream of $\CA$ customers is independent of the workload process of the main queue $\mathrm{Q_M}$ for all $t \ge 0$. Consequently, this departure stream constitutes a causally independent Poisson process, as defined in Definition~\ref{def:exp:causal_indep}, satisfies Wolff's lack of anticipation assumption~\cite{wolff1982poisson}, and permits the rigorous application of the PASTA property at its departure epochs.
\section{Proof of Lemma~\ref{lem:operator_properties}}
\label{app:operator_properties_proof}
We establish the properties of the operators $T_{\mathrm L}$ and $T_{\mathrm U}$ in four steps: preservation of the function space $\overline{\mathcal{F}}$, monotonicity, continuity under pointwise convergence, and uniqueness of fixed points.

\paragraph{Preservation of the space $\overline{\mathcal{F}}$}
We first extend the definition by setting $T_{\mathrm L}(\overline F)(x) = T_{\mathrm U}(\overline F)(x) = 1$ for $x < 0$. For $x \ge 0$, the integrals in \eqref{eq:TLdef}--\eqref{eq:TUdef} are non-negative. Bounding $\overline{F}(x - t) \le 1$ yields the uniform upper bound $\lambda \big(\E[X_{\CP}] + \E[X_{\CS}]\big) = \rho < 1$ for both operators in \eqref{eq:TLdef}--\eqref{eq:TUdef}. Thus, the outputs take values in $[0, \rho] \subset [0, 1]$ for $x \ge 0$. Since $\overline{F}$ is non-increasing, $x_1 \le x_2$ implies $\overline{F}(x_1 - t) \ge \overline{F}(x_2 - t)$, so both $T_{\mathrm L}(\overline F)$ and $T_{\mathrm U}(\overline F)$ are non-increasing on $x \in [0, \infty)$. Finally, as $x_k \downarrow x$, the right-continuity of $\overline{F}$ implies $\overline{F}(x_k - t) \to \overline{F}(x - t)$; because the integrands are bounded by their respective integrable service tails, the dominated convergence theorem (DCT) ensures that $T_{\mathrm L}(\overline F)$ and $T_{\mathrm U}(\overline F)$ are right-continuous on $x \in [0, \infty)$. Hence, both operators map $\overline{\mathcal{F}}$ into itself.

\paragraph{Monotonicity of the operators}
Let $\overline F_1, \overline F_2 \in \overline{\mathcal{F}}$ with $\overline F_1 \le \overline F_2$ pointwise. Because $\overline B_{\CP}$, $\overline B_{\CS}$, and $\overline B_{\CP+\CS}$ are non-negative, the inequality is preserved under integration:
\[
    \int_0^\infty \overline B(t) \overline F_1(x - t) \diff t \le \int_0^\infty \overline B(t) \overline F_2(x - t) \diff t.
\]
Consequently, $T_{\mathrm L}(\overline F_1) \le T_{\mathrm L}(\overline F_2)$ and $T_{\mathrm U}(\overline F_1) \le T_{\mathrm U}(\overline F_2)$, establishing monotonicity.

\paragraph{Continuity under pointwise convergence}
Suppose a sequence $\overline F_n \to \overline F$ pointwise as $n \to \infty$. For the corresponding service tails, the integrals are finite. By the DCT, the limit can be passed under the integral, proving that $T_{\mathrm L}(\overline F_n) \to T_{\mathrm L}(\overline F)$ and $T_{\mathrm U}(\overline F_n) \to T_{\mathrm U}(\overline F)$ pointwise on $[0, \infty)$.

\paragraph{Uniqueness of fixed points}
The equation $\overline F = T_{\mathrm L}(\overline F)$ is the level-crossing equation for the stationary workload tail of a stable two-class $\mathrm{M/G/1}$ queue with independent Poisson arrivals at aggregate rate $2\lambda$ and service-time distributions $B_{\CP}$ and $B_{\CS}$~\cite{Kleinrock1975Volume1,brill2017level}. Because $\rho < 1$, this queue admits a unique stationary workload distribution, and therefore $T_{\mathrm L}$ has a unique fixed point in $\overline{\mathcal{F}}$. Similarly, $\overline F = T_{\mathrm U}(\overline F)$ is the level-crossing equation for a batch-arrival $\mathrm{M^{[2]}/G/1}$ queue with batch size two and total service requirement $X_{\CP} + X_{\CS}$, which is likewise stable under $\rho < 1$ and possesses a unique fixed point in $\overline{\mathcal{F}}$.

\section{Second-Moment Analysis of Aggregate Workload Input}
\label{app:mgf}

We establish a uniform linear upper bound on the variance of the cumulative workload input $A(t)$ entering the system in stationarity over an interval of length $t > 0$, assuming only finite second moments of $(X_{\CP}, X_{\CS})$ and finite $\E[Y] < \infty$. The service times within each pair $(X_{\CP}, X_{\CS})$ may exhibit general dependence.

In stationarity, the cumulative workload input $A(t)$ entering the queue during $(0, t]$ decomposes into two mutually independent components:
\begin{equation}
    A(t) = A^\circ(t) + A^+(t),
\end{equation}
where:
\begin{itemize}
    \item $A^\circ(t)$ represents the work entering during $(0, t]$ from customers whose primary job arrived prior to $t = 0$, but whose secondary job arrives in $(0, t]$. The number of such customers present in $\mathrm{Q_A}$ at $t = 0$, denoted by $N$, is Poisson distributed with parameter $\gamma \coloneqq \lambda \E[Y]$. Each pending secondary independently completes its delay in $(0, t]$ with probability $F_Y^{\rm res}(t) \coloneqq \frac{1}{\E[Y]} \int_0^t \overline{F}_Y(s) \diff s$. Thus, the number of completing secondaries is Poisson distributed with parameter $\gamma F_Y^{\rm res}(t)$, and each contributes service requirement $X_{\CS}$. Hence, $A^\circ(t)$ is a compound Poisson sum with variance
    \begin{equation}
        \var(A^\circ(t)) = \gamma F_Y^{\rm res}(t) \, \E[X_{\CS}^2].
        \label{eq:var_A_circ}
    \end{equation}
    
  \item $A^+(t)$ represents the work entering from primary arrivals occurring in $(0, t]$ (a Poisson process with rate $\lambda$) and any of their corresponding secondary arrivals that enter before $t$. Each primary customer arrives at a uniform epoch $U \in (0, t)$ and brings service $X_{\CP}$. Its secondary customer also arrives before $t$ if and only if $U + Y \le t$, which occurs with probability $p = \frac{1}{t} \int_0^t F_Y(s) \diff s = 1 - \frac{\E[Y]}{t} F_Y^{\rm res}(t)$. By Poisson thinning, $A^+(t)$ is the sum of two independent compound Poisson sums:
    (i) primaries whose secondaries also arrive before $t$, with Poisson parameter $\lambda t p = \lambda t - \gamma F_Y^{\rm res}(t)$ and contributing total service $X_{\CP} + X_{\CS}$; and
    (ii) primaries whose secondaries arrive after $t$, with Poisson parameter $\lambda t (1-p) = \gamma F_Y^{\rm res}(t)$ and contributing only $X_{\CP}$.
    Thus,
    \begin{equation}
        \var(A^+(t)) = \left(\lambda t - \gamma F_Y^{\rm res}(t)\right) \E\left[(X_{\CP} + X_{\CS})^2\right] + \gamma F_Y^{\rm res}(t) \, \E[X_{\CP}^2].
        \label{eq:var_A_plus}
    \end{equation}
\end{itemize}

Summing \eqref{eq:var_A_circ} and \eqref{eq:var_A_plus}, the variance of the aggregate input $A(t)$ is
\begin{align}
    \var(A(t)) 
    &= \var(A^\circ(t)) + \var(A^+(t)) \nonumber \\
    &= \lambda t \, \E\left[(X_{\CP} + X_{\CS})^2\right] - \gamma F_Y^{\rm res}(t) \left( \E\left[(X_{\CP} + X_{\CS})^2\right] - \E[X_{\CP}^2] - \E[X_{\CS}^2] \right) \nonumber \\
    &= \lambda t \, \E\left[(X_{\CP} + X_{\CS})^2\right] - 2 \gamma F_Y^{\rm res}(t) \, \E[X_{\CP} X_{\CS}].
    \label{eq:exact_variance_general}
\end{align}
Because $X_{\CP}, X_{\CS} \ge 0$, we have $\E[X_{\CP}X_{\CS}] \ge 0$. Since $\gamma F_Y^{\rm res}(t) \ge 0$, dropping the non-negative subtractive term yields the uniform upper bound
\begin{equation}
    \var(A(t)) \le C t,
    \label{eq:variance_linear_bound}
\end{equation}
where $C \coloneqq \lambda \E\left[(X_{\CP} + X_{\CS})^2\right] < \infty$. Bound \eqref{eq:variance_linear_bound} holds universally across all triggering-delay distributions and service distributions with finite second moments.

\section{Proof of Theorem~\ref{thm:asymp_large_delay}}
\label{app:large_delay_proof}

We aim to demonstrate that, when the triggering delays diverge to infinity in probability ($Y^{(k)} \xrightarrow{\mathrm{P}} \infty$), the stationary workload converges in distribution to that of the two-class $\mathrm{M/G/1}$ queue (the lower bound): $V_{\mathrm M}^{(k)} \xrightarrow{\mathrm d} V_{\mathrm L}$. The underlying physical intuition is that large triggering delays cause secondary customers to arrive long after their primary counterparts, thereby decoupling the two streams. Formally, the theorem establishes convergence of the stationary workload distribution; it does not assert convergence of the entire arrival point process. To formalize this result, we show that $\Delta^{(k)}$ converges to zero in probability, which, in turn, drives the stationary workload distribution toward the lower-bound fixed point $V_{\mathrm L}$.

We begin by examining the workload observed by arriving secondary customers through the coupling construction of Section~\ref{sec:ph}. Let $W_{\CS}^{(k)}$ denote the workload observed by an arriving secondary customer in the $k$-th system. From Theorem~\ref{thm:ph:conditional_dist}, this workload satisfies
\begin{equation}\label{eq:asymp_WCS_formal}
    W_{\CS}^{(k)} \overset{\mathrm d}{=} V^{\mathrm{RP},(k)}_{\mathrm M}((Y^{(k)})^-) + \Delta^{(k)}, \qquad V^{\mathrm{RP},(k)}_{\mathrm M}((Y^{(k)})^-) \overset{\mathrm d}{=} V_{\mathrm{M}}^{(k)},
\end{equation}
Here, $\Delta^{(k)}$ depends on whether the queue has emptied at least once before the tagged secondary arrives. Let $\bar{A}^{(k)}(s)$ denote the cumulative workload arriving at $\mathrm{Q_M}$ during $(0,s]$. Let $T_0^{(k)}$ denote the first time the queue becomes empty after the arrival of the tagged primary customer:
\begin{equation}\label{eq:asymp_T0_def}
    T_0^{(k)} \coloneqq \inf \left\{ s > 0 : V_{\mathrm{M}}^{(k)} + X^{\star}_{\CP} + \bar{A}^{(k)}(s) - s \le 0 \right\}.
\end{equation}
Furthermore, define the interval infimum of the target workload trajectory over the triggering delay $(0, Y^{(k)})$ as
\begin{equation}\label{eq:asymp_I_def}
    I^{(k)}((Y^{(k)})^-) \coloneqq \inf_{s \in (0, Y^{(k)})} \left( V_{\mathrm{M}}^{(k)} + X^{\star}_{\CP} + \bar{A}^{(k)}(s) - s \right).
\end{equation}
By the Skorokhod reflection mapping analysis in Section~\ref{subsec:ph:coupling}, we have
\begin{equation}\label{eq:asymp_Delta_piecewise}
    \Delta^{(k)} = \min\left\{ X^{\star}_{\CP}, \max\bigl(0,I^{(k)}((Y^{(k)})^-)\bigr) \right\}.
\end{equation}
Observing that $I^{(k)}((Y^{(k)})^-) \le 0$ if and only if $Y^{(k)} \ge T_0^{(k)}$, \eqref{eq:asymp_Delta_piecewise} can be written equivalently as
\begin{equation}\label{eq:asymp_Delta_compact}
    \Delta^{(k)} = \mathbbm{1}_{\{Y^{(k)} < T_0^{(k)}\}} \min\left\{ X^{\star}_{\CP}, I^{(k)}((Y^{(k)})^-)\right\}.
\end{equation}

Next, we establish that the family of the \textit{first} clearing times $\{T_0^{(k)}\}_{k \ge 1}$ is uniformly tight. If $T_0^{(k)} > s$, then the workload has remained strictly positive throughout $[0, s]$, which implies that $V_{\mathrm M}^{(k)} + X_{\CP}^{\star} + \bar A^{(k)}(s) - s > 0$. Set $\delta \coloneqq (1-\rho)/2 > 0$, so that $s - \rho s = 2\delta s$. The union bound then yields
\begin{align}
    \Prob\bigl(T_0^{(k)} > s\bigr) 
    &\le \Prob\bigl(V_{\mathrm M}^{(k)} + X_{\CP}^{\star} + \bar A^{(k)}(s) - s > 0\bigr) \nonumber \\
    &= \Prob\bigl((V_{\mathrm M}^{(k)} + X_{\CP}^{\star}) + (\bar A^{(k)}(s) - \rho s) > 2\delta s\bigr) \nonumber \\
    &\le \Prob\bigl(V_{\mathrm M}^{(k)} + X_{\CP}^{\star} > \delta s\bigr) + \Prob\bigl(\bar A^{(k)}(s) - \rho s > \delta s\bigr).
    \label{eq:asymp_tightness_split}
\end{align}
The first term is bounded uniformly in $k$ by the upper bound of \eqref{eq:workload_bounds},
\[
    \Prob(V_{\mathrm M}^{(k)} + X_{\CP}^{\star} > \delta s) \le \Prob(V_{\mathrm U} + X_{\CP}^{\star} > \delta s).
\]
For the second term, applying Chebyshev's inequality alongside the uniform variance bound \eqref{eq:variance_linear_bound} from Appendix~\ref{app:mgf} gives
\[
    \Prob(\bar A^{(k)}(s) - \rho s > \delta s) \le \frac{\var(\bar A^{(k)}(s))}{\delta^2 s^2} \le \frac{C}{\delta^2 s}.
\]
Combining these bounds establishes the uniform tail bound:
\begin{equation}
    \sup_{k\ge1}\Prob\bigl(T_0^{(k)}>s\bigr) \le \Prob\bigl(V_{\mathrm U}+X_{\CP}^{\star}>\delta s\bigr) + \frac{C}{\delta^2 s} \xrightarrow[s\to\infty]{} 0,
\end{equation}
which proves that $\{T_0^{(k)}\}_{k\ge1}$ is uniformly tight. Because the tagged triggering delay is independent of the pre-arrival workload, the tagged service pair, and the future arrival and service primitives that determine $T_0^{(k)}$, it is independent of $T_0^{(k)}$. For every fixed $M > 0$, the set inclusion $\{Y^{(k)} < T_0^{(k)}\} \subseteq \{T_0^{(k)} > M\} \cup \{Y^{(k)} < M\}$ yields
\[
    \Prob(Y^{(k)} < T_0^{(k)}) \le \Prob(T_0^{(k)} > M) + \Prob(Y^{(k)} < M).
\]
Choosing $M$ sufficiently large ensures $\sup_{k\ge 1}\Prob(T_0^{(k)} > M) < \varepsilon/2$, while $Y^{(k)} \xrightarrow{\mathrm{P}} \infty$ implies $\Prob(Y^{(k)} < M) \to 0$ as $k \to \infty$. Thus $\lim_{k\to\infty}\Prob(Y^{(k)} < T_0^{(k)}) = 0$, which yields
\[
    \Delta^{(k)} \xrightarrow{\mathrm{P}} 0.
\]

To complete the proof, we apply subsequential compactness and identify the unique limit. Consider any subsequence, still indexed by $k$, such that $V_{\mathrm M}^{(k)} \xrightarrow{\mathrm d} V_{\mathrm M}^{(\infty)}$. By Slutsky's theorem (Cram\'er's theorem)~\cite[pp.~248--249]{gut2005probability},
\[
    \bigl(V_{\mathrm M}^{(k)},\Delta^{(k)}, X^\star_\CP, X^\star_\CS\bigr) \xrightarrow{\mathrm d} \bigl(V_{\mathrm M}^{(\infty)},0,X^\star_\CP, X^\star_\CS\bigr),
\]
and therefore $W_{\CS}^{(k)} \xrightarrow{\mathrm d} V_{\mathrm M}^{(\infty)}$, independent of $(X^\star_\CP, X^\star_\CS)$. Applying the exact workload identity \eqref{eq:palm_tail} to the $k$-th system, simplifying as in \eqref{eq:secondary_upper_conditional}, and passing to the limit via the DCT gives
\begin{align}
    \overline F_{V_{\mathrm M}^{(\infty)}}(x)
    &= \lim_{k\to\infty} \overline F_{V_{\mathrm M}^{(k)}}(x) \nonumber \\
    &= \lambda \E\left[\int_0^{X^\star_{\CP}} \Prob\bigl(V_{\mathrm M}^{(\infty)}+t>x\bigr)\diff t\right]
       + \lambda \E\left[\int_0^{X^\star_{\CS}} \Prob\bigl(V_{\mathrm M}^{(\infty)}+t>x\bigr)\diff t\right] \nonumber \\
    &= T_{\mathrm L}\bigl(\overline F_{V_{\mathrm M}^{(\infty)}}\bigr)(x).
    \label{eq:asymp_fixed_point_large}
\end{align}
Thus, $\overline F_{V_{\mathrm M}^{(\infty)}}$ is a fixed point of $T_{\mathrm L}$. By Lemma~\ref{lem:operator_properties}, this fixed point is unique in $\overline{\mathcal{F}}$ and equals $\overline F_{V_{\mathrm L}}$. Since every convergent subsequence converges to the same limit $V_{\mathrm L}$ and the family $\{V_{\mathrm M}^{(k)}\}$ is tight (from \eqref{eq:workload_bounds}), the entire sequence converges in distribution:
\begin{equation}
    V_{\mathrm M}^{(k)} \xrightarrow{\mathrm d} V_{\mathrm L}
    \qquad \text{as } k\to\infty.
\end{equation}
This completes the proof.

\section{Proof of Theorem~\ref{thm:asymp_zero_delay}}
\label{app:small_delay_proof}

We aim to prove that when triggering delays converge to zero in probability ($Y^{(k)} \xrightarrow{\mathrm{P}} 0$), the stationary workload converges in distribution to the batch-arrival $\mathrm{M^{[2]}/G/1}$ queue (upper bound): $V_{\mathrm M}^{(k)} \xrightarrow{\mathrm d} V_{\mathrm U}$. The physical intuition is that near-instantaneous triggering causes secondary customers to enter the queue while their primary counterpart is still in the queue, making the arrival stream behave as a batch process. We formalize this by showing that $\Delta^{(k)}$ converges in probability to the primary service time $X_{\CP}^{\star}$, leading to the batch-arrival fixed point $V_{\mathrm U}$.

We first establish the asymptotic behavior of  $\Delta^{(k)}$. Under the stability condition \eqref{eq:model:stability}, the bound $V_{\mathrm M}^{(k)} \le_{\mathrm{st}} V_{\mathrm U}$ ensures that $\{V_{\mathrm M}^{(k)}\}_{k\ge1}$ is uniformly tight, so every subsequence possesses a weakly convergent subsubsequence by Prohorov's theorem. Consider the tagged primary customer arriving at $t = 0$ with strictly positive service requirement $X_{\CP}^{\star} > 0$ a.s., whose system time satisfies $S_{\CP}^{(k)} = V_{\mathrm M}^{(k)} + X_{\CP}^{\star} \ge X_{\CP}^{\star} > 0$ a.s. Since $Y^{(k)} \xrightarrow{\mathrm{P}} 0$, the probability that the primary departs before its secondary arrives vanishes:
\[
\Prob\bigl(S_{\CP}^{(k)} \le Y^{(k)}\bigr) \le \Prob\bigl(X_{\CP}^{\star} \le Y^{(k)}\bigr) \to 0 \qquad \text{as } k \to \infty.
\]
Thus, the primary is still in the queue at time $Y^{(k)}$ with probability tending to $1$. On the event $\{S_{\CP}^{(k)} > Y^{(k)}\}$, the queue has not emptied during $(0, Y^{(k)})$, and the primary has received at most $Y^{(k)}$ units of service, yielding the sample-path bound $|\Delta^{(k)} - X_{\CP}^{\star}| \le Y^{(k)}$. Therefore, for every $\varepsilon > 0$,
\[
\Prob\Bigl(\bigl|\Delta^{(k)} - X_{\CP}^{\star}\bigr| > \varepsilon\Bigr) \le \Prob\bigl(Y^{(k)} > \varepsilon\bigr) + \Prob\bigl(S_{\CP}^{(k)} \le Y^{(k)}\bigr) \to 0,
\]
which proves that
\begin{equation}
\Delta^{(k)} \xrightarrow{\mathrm{P}} X_{\CP}^{\star} \qquad \text{as } k \to \infty.
\end{equation}

Next, we characterize the limiting workload distribution. Let $\{V_{\mathrm M}^{(k)}\}_{k\ge1}$ be any subsequence converging in distribution to a limit $V_{\mathrm M}^{(\infty)}$. By Slutsky's theorem, 
\[
    \bigl(V_{\mathrm M}^{(k)}, \Delta^{(k)}, X^\star_\CP, X^\star_\CS\bigr) \xrightarrow{\mathrm d} \bigl(V_{\mathrm M}^{(\infty)}, X^\star_\CP, X^\star_\CP, X^\star_\CS\bigr),
\]
which implies that the workload seen by secondary customers satisfies $W_{\CS}^{(k)} \xrightarrow{\mathrm d} V_{\mathrm M}^{(\infty)}+X_{\CP}^{\star}$. Applying the exact workload identity \eqref{eq:palm_tail} to the $k$-th system, simplifying the secondary term as in \eqref{eq:secondary_upper_conditional}, and passing to the limit via the DCT gives
\begin{align*}
\overline F_{V_{\mathrm M}^{(\infty)}}(x) 
&= \lambda \E\Bigl[ \int_0^{X_{\CP}^{\star}} \mathbbm{1}_{\{V_{\mathrm M}^{(\infty)}+t>x\}} \diff t \Bigr] + \lambda \E\Bigl[ \int_0^{X_{\CS}^{\star}} \mathbbm{1}_{\{V_{\mathrm M}^{(\infty)}+X_{\CP}^{\star}+t>x\}} \diff t \Bigr] \\
&= \lambda \E\Bigl[\int_0^{X_{\CP}^{\star}+X_{\CS}^{\star}}  \mathbbm{1}_{\{V_{\mathrm M}^{(\infty)}+u>x\}}  \diff u \Bigr] = T_{\mathrm U}\bigl(\overline F_{V_{\mathrm M}^{(\infty)}}\bigr)(x), \qquad x\ge0.
\end{align*}
By Lemma~\ref{lem:operator_properties}, the fixed point of $T_{\mathrm U}$ is unique in $\overline{\mathcal{F}}$ and equals $\overline F_{V_{\mathrm U}}$, so $V_{\mathrm M}^{(\infty)} \overset{\mathrm d}{=} V_{\mathrm U}$. Because every convergent subsequence converges to the same unique limit $V_{\mathrm U}$ and the sequence is tight, the entire sequence converges in distribution:
\begin{equation}
V_{\mathrm M}^{(k)} \xrightarrow{\mathrm d} V_{\mathrm U} \qquad\text{as } k\to\infty.
\end{equation}
This completes the proof.

\bibliographystyle{elsarticle-num}
\bibliography{References}

\end{document}